\documentclass{amsproc}
\usepackage{breqn,float,amssymb,pdflscape,rotating}
\makeatletter
\@namedef{subjclassname@2020}{%
  \textup{2020} Mathematics Subject Classification}
\makeatother
\newtheorem{theorem}{Theorem}[section]

\theoremstyle{definition}

\newtheorem{example}[theorem]{Example}

\theoremstyle{remark}

\numberwithin{equation}{section}
\allowdisplaybreaks
\begin{document}

\title{Extended Levett trigonometric series}


\author{Robert Reynolds}
\address[Robert Reynolds]{Department of Mathematics and Statistics, York University, Toronto, ON, Canada, M3J1P3}
\email[Corresponding author]{milver73@gmail.com}
\thanks{}




\date{}

\dedicatory{}

\begin{abstract}
An extension of two finite trigonometric series is studied to derive closed form formulae involving the Hurwitz-Lerch zeta function. The trigonometric series involves angles with a geometric series involving the powers of 3. These closed formulae are used to derive composite finite and infinite series involving special functions, trigonometric functions and fundamental constants. A short table summarizing some interesting results is produced.
\end{abstract}

\maketitle
\section{History}
%
%
RAWDON LEVETT (February 2, 1837 - February 22, 1914.), as aptly noted in The Times, \cite{mayo}, was a distinguished educator, ranking among the most esteemed teachers of his era. He passed away at the age of seventy-nine in his residence in Colwyn Bay in February. Following his graduation as the Eleventh Wrangler in 1865, he briefly taught at Rossall before assuming the role of a mathematics instructor at King Edward's School in Birmingham. It was here that he dedicated his career, ultimately becoming the Head of the Mathematics Department. Despite being offered the position of Head Master, he declined due to health concerns. Approximately two decades ago, he retired from his teaching profession. With his passing, the last of a remarkable trio of educators Vardy, Hunter Smith, and Levett faded away. These individuals not only left their indelible mark on their respective schools (which had previously nurtured prominent figures like Lightfoot, Westcott, and Benson) but also contributed significantly to the honorable legacy they inherited. In the 19th century, Rawdon Levett made a significant contribution to education, particularly in the area of mathematical training. He was honoured for his skill as a teacher and dedication to his trade. At Birmingham's King Edward's School, Levett taught mathematics and had a lasting impression on his students. His legacy includes his commitment to teaching and his influence on education.\\\\
\section{Introduction}
In 1892 Rawdon Levett, published his landmark book entitled "The elements of plane geometry" \cite{levett}, which was not written as a text-book, but rather as a detailed syllabus of proofs, which is divided into three sections. These sections respectively address the concepts related to arithmetic, real algebra, and complex quantities. This arrangement is quite natural and offers the advantage of introducing new terminology and formulas associated with the subject before students encounter the challenge of applying signs to represent the direction and orientation of lines. Part I is made even simpler by deferring the discussion of circular angle measurement. For various practical applications, such as surveying and basic mechanics involving trigonometry, the concise introduction provided in Part I can prove to be valuable to interested readers.\\\\
The focus of this research is in chapter XIV dedicated to Factors with respect to the fundamental theorem on trigonometrical factors. In this section of the book Levett produces many exciting finite and infinite trigonometric series. We focus on two obscure but interesting  finite series involving  secant-sine and cosecant-cosine series.
In this current paper we applied a contour integral method to two obscure finite trigonometric series tabled in the book by Levett \cite{levett}. The forms of these two series are of particular interest as the angle in these series are of a geometric series form involving 3 raised to an integer power. These types of formulae have been studied by the author in previous work with different geometric angular forms. In this work we will look at the following finite series forms; (i) Finite sum of Hurwitz-Lerch zeta function, (ii) Finite sum of composite special functions, (iii) Finite products of trigonometric functions, (iv) Finite products of quotient gamma functions, (v) Quotient trigonometric functions with reciprocal angles and (vi)  Table of interesting results. Some of the finite and infinite products involving the gamma function can be simplified using the Gauss's multiplication formula \cite{pearson}.\\\\ 
A special function is a mathematical function, typically named after an early researcher who studied its properties, and it serves a specific purpose in the area of mathematical physics or various branches of mathematics. Notable instances of such functions comprise the gamma function, hypergeometric function, Whittaker function, Hurwitz-Lerch zeta and Meijer G-function. Applications of special functions are prevalent in various areas of mathematics and related fields, such as engineering, quantum physics, astronomy and combinatorics. Finite sums of trigonometric functions are listed in \cite{grad}, \cite{hansen}, chapter 4  in \cite{prud1}, and page 131 in \cite{davis}. Finite sums of special functions are tabled in \cite{hansen} and chapter 4 in \cite{prud2}. \\\\
In this work we apply the contour integral method \cite{reyn4}, to the finite secant-sine and cosecant-cosine sums given in Examples XXII, numbers (50), (48) and (62) on page 335 in \cite{levett} respectively to derive their contour integral finite trigonometric forms, resulting in
\begin{multline}\label{eq:contour}
\frac{1}{2\pi i}\int_{C}\sum_{p=0}^{n-1}3^{-p} a^w w^{-k-1} \sin ^3\left(3^p (m+w)\right) \sec \left(3^{p+1} (m+w)\right)dw\\
=\frac{1}{2\pi i}\int_{C}\frac{3}{8} a^w w^{-k-1} \left(3^{-n}
   \tan \left(3^n (m+w)\right)-\tan (m+w)\right)dw
\end{multline}
where $a,m,k\in\mathbb{C},Re(m+w)>0,n\in\mathbb{Z^{+}}$. Using equation (\ref{eq:contour}) the main Theorem to be derived and evaluated is given by
\begin{multline}
\sum_{p=0}^{n-1}\left(i 3^{p+1}\right)^k e^{i m 3^p} \left(\Phi \left(e^{2 i 3^{p+1} m},-k,\frac{1}{6} \left(1-i 3^{-p} \log (a)\right)\right)\right. \\ \left.
+e^{4 i m 3^p} \Phi \left(e^{2 i 3^{p+1}
   m},-k,\frac{1}{6} \left(5-i 3^{-p} \log (a)\right)\right)\right)\\
=i^k e^{i m} \Phi \left(e^{2 i m},-k,\frac{1}{2}-\frac{1}{2} i \log (a)\right)-\left(i 3^n\right)^k e^{i m 3^n} \Phi
   \left(e^{2 i 3^n m},-k,\frac{1}{2}-\frac{1}{2} i 3^{-n} \log (a)\right)
\end{multline}
where the variables $k,a,m$ are general complex numbers and $n$ are positive integers. The second contour integral form used is given by;
\begin{multline}\label{eq:contour1}
\frac{1}{2\pi i}\int_{C}\sum_{p=0}^{n-1}a^w w^{-k-1} \cos \left(2\times 3^p (m+w)\right) \csc \left(3^{p+1} (m+w)\right)dw\\
=\frac{1}{2\pi i}\int_{C}\frac{1}{2} a^w w^{-k-1} \csc (m+w)-\frac{1}{2}
   a^w w^{-k-1} \csc \left(3^n (m+w)\right)dw
\end{multline}
where $a,m,k\in\mathbb{C},Re(m+w)>0,n\in\mathbb{Z^{+}}$. Using equation (\ref{eq:contour1}) the main Theorem to be derived and evaluated is given by
\begin{multline}
\sum_{p=0}^{n-1}3^{-p} \left(\log ^k(a)+2^k \left(i 3^{p+1}\right)^k e^{2 i m 3^p} \left(-3 \Phi \left(-e^{2 i 3^{p+1}
   m},-k,\frac{1}{6} \left(2-i 3^{-p} \log (a)\right)\right)\right.\right. \\ \left.\left.
+3 e^{2 i m 3^p} \Phi \left(-e^{2 i 3^{p+1}
   m},-k,\frac{1}{6} \left(4-i 3^{-p} \log (a)\right)\right)\right.\right. \\ \left.\left.
   -2 e^{4 i m 3^p} \Phi \left(-e^{2 i 3^{p+1}
   m},-k,\frac{1}{6} \left(6-i 3^{-p} \log (a)\right)\right)\right)\right)\\
=\frac{1}{2} 3^{1-n}
   \left(\left(3^n-1\right) \log ^k(a)+2^{k+1} \left(\left(i 3^n\right)^k e^{2 i m 3^n} \Phi \left(-e^{2 i 3^n
   m},-k,1-\frac{1}{2} i 3^{-n} \log (a)\right)\right.\right. \\ \left.\left.-i^k e^{2 i m} 3^n \Phi \left(-e^{2 i m},-k,1-\frac{1}{2} i \log
   (a)\right)\right)\right)
\end{multline}
where the variables $k,a,m$ are general complex numbers and $n$ are positive integers.
The third contour integral form used is given by;
\begin{multline}
\frac{1}{2\pi i}\int_{C}\sum_{p=0}^{n-1}3^{-p} a^w w^{-k-1} \sin \left(3^p (m+w)\right) \left(2 \cos \left(2\times 3^p (m+w)\right)\right. \\ \left.
+3^{p+1}-2\right) \sec \left(3^{p+1} (m+w)\right)dw\\
=\frac{1}{2\pi i}\int_{C}\frac{1}{2} 3^{1-n} \left(3^n-1\right) a^w
   w^{-k-1} \tan \left(3^n (m+w)\right)dw
\end{multline}
where $a,m,k\in\mathbb{C},Re(m+w)>0,n\in\mathbb{Z^{+}}$. Using equation (\ref{eq:contour1}) the main Theorem to be derived and evaluated is given by
\begin{multline}
\sum_{p=0}^{n-1}3^{-p} \left(i 3^{p+1}\right)^k e^{2 i m 3^p} \left(3 \left(3^p-1\right) \Phi \left(-e^{2 i 3^{p+1} m},-k,\frac{1}{6} \left(3^{-p} a+2\right)\right)\right. \\ \left.
-3 \left(3^p-1\right) e^{2 i m 3^p} \Phi \left(-e^{2 i 3^{p+1}m},-k,\frac{1}{6} \left(3^{-p} a+4\right)\right)\right. \\ \left.
-2 e^{4 i m 3^p} \Phi \left(-e^{2 i 3^{p+1} m},-k,\frac{1}{6}\left(3^{-p} a+6\right)\right)\right)\\
=-3 i \left(3^n-1\right) \left(i 3^n\right)^{k-1} e^{2 i m 3^n} \Phi
   \left(-e^{2 i 3^n m},-k,\frac{3^{-n} a}{2}+1\right)
\end{multline}
where the variables $k,a,m$ are general complex numbers and $n$ are positive integers.
The derivations follow the method used by us in \cite{reyn4}. This method involves using a form of the generalized Cauchy's integral formula given by
\begin{equation}\label{intro:cauchy}
\frac{y^k}{\Gamma(k+1)}=\frac{1}{2\pi i}\int_{C}\frac{e^{wy}}{w^{k+1}}dw,
\end{equation}
where $y,w\in\mathbb{C}$ and $C$ is in general an open contour in the complex plane where the bilinear concomitant \cite{reyn4} is equal to zero at the end points of the contour. This method involves using a form of equation (\ref{intro:cauchy}) then multiplies both sides by a function, then takes the double finite sum of both sides. This yields a double finite sum in terms of a contour integral. Then we multiply both sides of equation (\ref{intro:cauchy}) by another function and take the infinite sum of both sides such that the contour integral of both equations are the same.
\section{The Hurwitz-Lerch zeta function}
We use equation (1.11.3) in \cite{erd} where $\Phi(z,s,v)$ is the Lerch function which is a generalization of the Hurwitz zeta $\zeta(s,v)$ and Polylogarithm functions $Li_{n}(z)$. In number theory and complex analysis, the Lerch function is a mathematical function that appears in many branches of mathematics and physics. It is named after Czech mathematician Mathias Lerch, who published a paper about the function in 1887. Numerous areas of mathematics, including number theory (especially in the investigation of the Riemann zeta function and its generalizations), complex analysis, and theoretical physics, all have uses for it. It can be used to express a variety of complex functions and series and is involved in numerous mathematical identities. The Lerch function has a series representation given by
\begin{equation}\label{knuth:lerch}
\Phi(z,s,v)=\sum_{n=0}^{\infty}(v+n)^{-s}z^{n}
\end{equation}
where $|z|<1, v \neq 0,-1,-2,-3,..,$ and is continued analytically by its integral representation given by
\begin{equation}\label{knuth:lerch1}
\Phi(z,s,v)=\frac{1}{\Gamma(s)}\int_{0}^{\infty}\frac{t^{s-1}e^{-(v-1)t}}{e^{t}-z}dt
\end{equation}
where $Re(v)>0$, and either $|z| \leq 1, z \neq 1, Re(s)>0$, or $z=1, Re(s)>1$.
\section{Derivation of generalized trigonometric contour integral representations}
In this section we will derive the contour integral representations for Theorems (\ref{eq:theorem_ss}) and (\ref{eq:theorem_cc}) respectively. 
\subsection{The Hurwitz-Lerch zeta function representation for the secant-sine function contour integral}
We use the method in \cite{reyn4}. Using a generalization of Cauchy's integral formula (\ref{intro:cauchy}) we first replace $y$ by $ \log (a)+i x+y$ then multiply both sides by $e^{mxi}$ then form a second equation by replacing $x$ by $-x$ and subtract both equations to get
\begin{multline}
\frac{e^{-i m x} \left((\log (a)-i x+y)^k-e^{2 i m x} (\log (a)+i x+y)^k\right)}{\Gamma(k+1)}\\
=-\frac{1}{2\pi i}\int_{C}2 i a^w w^{-k-1} e^{w y} \sin (x(m+w))dw
\end{multline}
Next we replace $y$ by $i b (2 y+1)$ and multiply both sides by $(-1)^y e^{i b m (2 y+1)}$ and take the infinite sum over $y\in[0,\infty)$ and simplify in terms of the Hurwitz-Lerch zeta function to get
\begin{multline}\label{gen_sec_sin}
\frac{i 2^k (i b)^k e^{i m (b-x)} \left(\Phi \left(-e^{2 i b m},-k,\frac{b-x-i \log (a)}{2 b}\right)-e^{2 i m x} \Phi
   \left(-e^{2 i b m},-k,\frac{b+x-i \log (a)}{2 b}\right)\right)}{\Gamma(k+1)}\\
   =-\frac{1}{2\pi i}\sum_{y=0}^{\infty}\int_{C}2 i(-1)^y a^w w^{-k-1} e^{i b (2 y+1) (m+w)} \sin (x (m+w))dw\\
    =-\frac{1}{2\pi i}\int_{C}\sum_{y=0}^{\infty}2 i(-1)^y a^w w^{-k-1} e^{i b (2 y+1) (m+w)} \sin (x (m+w))dw\\
   =\frac{1}{2\pi i}\int_{C}a^w w^{-k-1} \sec (b (m+w)) \sin (x (m+w))dw
   \end{multline}
from equation (1.232.2) and (1.411.1) in \cite{grad} where $Re(w+m)>0$ and $Im\left(m+w\right)>0$ in order for the sums to converge. We apply Tonelli's theorem for sums and integrals, see page 177 in \cite{gelca} as the summand and integral are of bounded measure over the space $\mathbb{C} \times [0,\infty)$.

\subsection{The Hurwitz-Lerch zeta function in terms of the tangent contour integral representation}
\subsubsection{Derivation of the additional contour}
Using a generalization of Cauchy's integral formula (\ref{intro:cauchy}), first replace $y \to \log (a)$ then multiply both sides by $-i $ and  simplify to get
\begin{equation}\label{eq:gen_lerch1}
-\frac{i \log ^k(a)}{\Gamma(k+1)}=-\frac{1}{2\pi i}\int_{C}i a^w w^{-k-1}dw
\end{equation}
Using a generalization of Cauchy's integral formula (\ref{intro:cauchy}), first replace $y \to \log (a)+2 i b (y+1)$ then multiply both sides by $-2 i (-1)^y e^{2 i b m (y+1)}$ and take the infinite sums over $y\in[0,\infty)$ and simplify in terms of the Hurwitz-Lerch zeta function and subtract equation (\ref{eq:gen_lerch1}) to get
\begin{multline}\label{gen_tan}
\frac{i \left(\log ^k(a)-2^{k+1} (i b)^k e^{2 i b m} \Phi \left(-e^{2 i b m},-k,1-\frac{i \log (a)}{2 b}\right)\right)}{\Gamma(k+1)}\\
=\frac{1}{2\pi i}\sum_{y=0}^{\infty}\int_{C}(-1)^y a^w w^{-k-1} e^{2 i b (y+1) (m+w)}dw\\
   =-\frac{1}{2\pi i}\int_{C}\sum_{y=0}^{\infty}(-1)^y a^w w^{-k-1} e^{2 i b (y+1) (m+w)}dw\\
=\frac{1}{2\pi i}\int_{C}a^w w^{-k-1} \tan (b (m+w))dw
\end{multline}
from equation (1.232.1) in \cite{grad} where $Re\left(  m+w \right)>0$ and $Im\left(m+w\right)>0$ in order for the sums to converge. Apply Tonelli's theorem for multiple sums, see page 177 in \cite{gelca} as the summands are of bounded measure over the space $\mathbb{C} \times [0,\infty)$.\\\\
%
%
\subsection{Derivation of the generalized cosecant-cosine contour integral}
We use the method in \cite{reyn4}. Using a generalization of Cauchy's integral formula (\ref{intro:cauchy}) we first replace $y$ by $ \log (a)+i x+y$ then multiply both sides by $e^{mxi}$ then form a second equation by replacing $x$ by $-x$ and add both equations to get
\begin{multline}
\frac{e^{-i m x} \left(e^{2 i m x} (\log (a)+i x+y)^k+(\log (a)-i
   x+y)^k\right)}{\Gamma(k+1)}\\
   =\frac{1}{2\pi i}\int_{C}2 w^{-k-1} e^{w (\log (a)+y)} \cos (x (m+w))dw
\end{multline}
Next we replace $y$ by $b I (2 y + 1)$ and multiply both sides by $-i e^{i b m (2 y+1)}$ and take the infinite and finite sums over $y\in[0,\infty)$ and simplify in terms of the Hurwitz-Lerch zeta function to get
\begin{multline}\label{eq:gen_cc_ci}
-\frac{i 2^k (i b)^k e^{i m (b-x)} \left(\Phi \left(e^{2 i b m},-k,\frac{b-x-i \log (a)}{2 b}\right)+e^{2 i m x} \Phi
   \left(e^{2 i b m},-k,\frac{b+x-i \log (a)}{2 b}\right)\right)}{\Gamma(k+1)}\\
   =\frac{1}{2\pi i}\sum_{y=0}^{\infty}\int_{C}2 a^w w^{-k-1} e^{i b (2 y+1) (m+w)} \cos (x (m+w))dw\\
   =\frac{1}{2\pi i}\int_{C}\sum_{y=0}^{\infty}2 a^w w^{-k-1} e^{i b (2 y+1) (m+w)} \cos (x (m+w))dw\\
   =\frac{1}{2\pi i}\int_{C}a^w w^{-k-1} \csc (b (m+w)) \cos (x (m+w))dw
\end{multline}
from equation (1.232.3) and (1.411.3) in \cite{grad} where $Re(w+m)>0$ and $Im\left(m+w\right)>0$ in order for the sums to converge. We apply Tonelli's theorem for sums and integrals, see page 177 in \cite{gelca} as the summand and integral are of bounded measure over the space $\mathbb{C} \times [0,\infty)$.
\subsection{Derivation of the generalized cosecant contour integral}
We use the method in \cite{reyn4}. Using equation (\ref{intro:cauchy})  we first replace $\log (a)+i b (2 y+1)$ and multiply both sides by $-2 i e^{i b m (2 y+1)}$ then take the infinite sum over $y\in [0,\infty)$ and simplify in terms of the Hurwitz-Lerch zeta function to get
\begin{multline}\label{eq:gen_csc_ci}
-\frac{i 2^{k+1} (i b)^k e^{i b m} \Phi \left(e^{2 i b m},-k,\frac{b-i
   \log (a)}{2 b}\right)}{\Gamma(k+1)}\\
   =\frac{1}{2\pi i}\sum_{y=0}^{\infty}\int_{C}a^w w^{-k-1} e^{i b (2 y+1) (m+w)}dw\\
   =\frac{1}{2\pi i}\int_{C}\sum_{y=0}^{\infty}a^w w^{-k-1} e^{i b (2 y+1) (m+w)}dw\\
   =\frac{1}{2\pi i}\int_{C}a^w w^{-k-1} \csc (b (m+w))dw
\end{multline}
from equation (1.232.3) in \cite{grad} where $Re(w+m)>0$ and $Im\left(m+w\right)>0$ in order for the sums to converge. We apply Tonelli's theorem for multiple sums, see page 177 in \cite{gelca} as the summands are of bounded measure over the space $\mathbb{C} \times [0,\infty)$.
\section{Derivation of the Hurwitz-Lerch zeta function contour integrals}
In this section we will derive the contour integrals by simple substitution in the previous contour integral representations.
\subsection{The secant-sine series}
\subsubsection{Left-hand side first contour integral}
Use equation (\ref{gen_sec_sin}) and replace $b$ by $3^{p+1}$ and $x$ by $3^p$ then multiply both sides by $3^{1-p}$ and take the finite sum over $p\in[0,n-1]$ to get;
\begin{multline}\label{lfci}
\sum_{p=0}^{n-1}\frac{1}{\Gamma(k+1)}i 2^k 3^{1-p} \left(i 3^{p+1}\right)^k e^{i m \left(3^{p+1}-3^p\right)} \\\left(\Phi \left(-e^{2 i 3^{p+1}m},-k,\frac{1}{2} 3^{-p-1} \left(-i \log (a)-3^p+3^{p+1}\right)\right)\right. \\ \left.
   -e^{2 i m 3^p} \Phi \left(-e^{2 i 3^{p+1} m},-k,\frac{1}{2}3^{-p-1} \left(-i \log (a)+3^p+3^{p+1}\right)\right)\right)\\
   =\frac{1}{2\pi i}\int_{C}\sum_{p=0}^{n-1}3^{1-p} a^w w^{-k-1} \sin \left(3^p (m+w)\right) \sec
   \left(3^{p+1} (m+w)\right)dw
\end{multline}
\subsubsection{Left-hand side second contour integral}
Use equation (\ref{gen_tan}) and replace $b$ by $3^{p+1}$ then multiply both sides by $-3^{-p}$ and take the finite sum over $p\in[0,n-1]$ to get;
\begin{multline}\label{lsci}
-\sum_{p=0}^{n-1}\frac{i 3^{-p} \left(\log ^k(a)-2^{k+1} \left(i 3^{p+1}\right)^k e^{2 i m 3^{p+1}} \Phi \left(-e^{2 i 3^{p+1}
   m},-k,1-\frac{1}{2} i 3^{-p-1} \log (a)\right)\right)}{\Gamma(k+1)}\\
   =-\frac{1}{2\pi i}\int_{C}\sum_{p=0}^{n-1}3^{-p} a^w w^{-k-1} \tan \left(3^{p+1} (m+w)\right)dw
\end{multline}
\subsubsection{Right-hand side first contour integral}
Use equation (\ref{gen_tan}) and replace $b$ by $1$ then multiply both sides by $-3/2$ to get;
\begin{multline}\label{rfci}
-\frac{3 i \left(\log ^k(a)-i^k 2^{k+1} e^{2 i m} \Phi \left(-e^{2 i m},-k,1-\frac{1}{2} i \log (a)\right)\right)}{2
   \Gamma(k+1)}\\
   =-\frac{1}{2\pi i}\int_{C}\frac{3}{2} a^w w^{-k-1} \tan (m+w)dw
\end{multline}
\subsubsection{Right-hand side first contour integral}
Use equation (\ref{gen_tan}) and replace $b$ by $3^{n}$ then multiply both sides by $\frac{3^{1-n}}{2}$ to get;
\begin{multline}\label{rsci}
\frac{i 3^{1-n} \left(\log ^k(a)-2^{k+1} \left(i 3^n\right)^k e^{2 i m 3^n} \Phi \left(-e^{2 i 3^n m},-k,1-\frac{1}{2} i 3^{-n}
   \log (a)\right)\right)}{2 \Gamma(k+1)}\\
   =\frac{1}{2\pi i}\int_{C}\frac{1}{2} 3^{1-n} a^w w^{-k-1} \tan \left(3^n (m+w)\right)dw
\end{multline}
\subsection{The cosecant-cosine series}
\subsubsection{Left-hand side contour integral}
Use equation (\ref{eq:gen_cc_ci}) and replace $b$ by $3^{p+1}$ and $x$ by $2\times 3^p$ and take the finite sum over $p\in[0,n-1]$ to get;
\begin{multline}\label{lfci_cc}
-\sum_{p=0}^{n-1}\frac{1}{\Gamma(k+1)}i 2^k \left(i 3^{p+1}\right)^k e^{i m \left(3^{p+1}-2\times 3^p\right)}\\
 \left(\Phi \left(e^{2 i 3^{p+1} m},-k,\frac{1}{2}
   3^{-p-1} \left(-i \log (a)-2\times 3^p+3^{p+1}\right)\right)\right. \\ \left.
   +e^{4 i m 3^p} \Phi \left(e^{2 i 3^{p+1} m},-k,\frac{1}{2} 3^{-p-1}
   \left(-i \log (a)+2\times 3^p+3^{p+1}\right)\right)\right)\\
   =\frac{1}{2\pi i}\int_{C}\sum_{p=0}^{n-1}a^w w^{-k-1} \cos \left(2\times 3^p (m+w)\right) \csc \left(3^{p+1}
   (m+w)\right)dw
\end{multline}
\subsubsection{Right-hand side first contour integral}
Use equation (\ref{eq:gen_csc_ci}) and replace $b$ by $1$ then multiply both sides by $1/2$ to get;
\begin{multline}\label{rfci_cc}
-\frac{i^{k+1} 2^k e^{i m} \Phi \left(e^{2 i m},-k,\frac{1}{2} (1-i \log (a))\right)}{\Gamma(k+1)}=\frac{1}{2\pi i}\int_{C}\frac{1}{2} a^w w^{-k-1} \csc(m+w)dw
\end{multline}
\subsubsection{Right-hand side second contour integral}
Use equation (\ref{eq:gen_csc_ci}) and replace $b$ by $3^n$ then multiply both sides by $-1/2$ to get;
\begin{multline}\label{rsci_cc}
\frac{i 2^k \left(i 3^n\right)^k e^{i m 3^n} \Phi \left(e^{2 i 3^n m},-k,\frac{1}{2} 3^{-n} \left(3^n-i \log
   (a)\right)\right)}{\Gamma(k+1)}\\
   =-\frac{1}{2\pi i}\int_{C}\frac{1}{2} a^w w^{-k-1} \csc \left(3^n (m+w)\right)dw
\end{multline}
\subsection{The sine-secant series}
\subsubsection{Left-hand side first contour integral}
Use equation (\ref{gen_sec_sin}) and replace $b$ by $3^{p+1}$ and $x$ by $3^p$ and multiply both sides by 3 and simplify to get;
\begin{multline}\label{lhfc_ss}
\sum_{p=0}^{n-1}\frac{1}{k!}3 i 2^k \left(i 3^{p+1}\right)^k e^{i m \left(3^{p+1}-3^p\right)}\\ \left(\Phi \left(-e^{2 i 3^{p+1} m},-k,\frac{1}{2} 3^{-p-1} \left(-i \log (a)-3^p+3^{p+1}\right)\right)\right. \\ \left.
-e^{2 i
   m 3^p} \Phi \left(-e^{2 i 3^{p+1} m},-k,\frac{1}{2} 3^{-p-1} \left(-i \log (a)+3^p+3^{p+1}\right)\right)\right)\\
   =\frac{1}{2\pi i}\int_{C}\sum_{p=0}^{n-1}3 a^w w^{-k-1} \sin \left(3^p (m+w)\right) \sec \left(3^{p+1}
   (m+w)\right)dw
\end{multline}
\subsubsection{Left-hand side second contour integral}
Use equation (\ref{gen_sec_sin}) and replace $b$ by $3^{p+1}$ and $x$ by $3^p$ and multiply both sides by $-3^{1-p}$ and simplify to get;
\begin{multline}\label{lhsc_ss}
-\sum_{p=0}^{n-1}\frac{1}{k!}i 2^k 3^{1-p} \left(i 3^{p+1}\right)^k e^{i m \left(3^{p+1}-3^p\right)}\\
 \left(\Phi \left(-e^{2 i 3^{p+1} m},-k,\frac{1}{2} 3^{-p-1} \left(-i \log
   (a)-3^p+3^{p+1}\right)\right)\right. \\ \left.
   =-e^{2 i m 3^p} \Phi \left(-e^{2 i 3^{p+1} m},-k,\frac{1}{2} 3^{-p-1} \left(-i \log (a)+3^p+3^{p+1}\right)\right)\right)\\
   =-\frac{1}{2\pi i}\int_{C}\sum_{p=0}^{n-1}3^{1-p} a^w w^{-k-1} \sin
   \left(3^p (m+w)\right) \sec \left(3^{p+1} (m+w)\right)dw
\end{multline}
\subsubsection{Left-hand side third contour integral}
Use equation (\ref{gen_tan}) and replace $b$ by $3^{p+1}$ and multiply both sides by $3^{-p}$ and simplify to get;
\begin{multline}\label{lhtc_ss}
\sum_{p=0}^{n-1}\frac{i 3^{-p} \left(\log ^k(a)-2^{k+1} \left(i 3^{p+1}\right)^k e^{2 i m 3^{p+1}} \Phi \left(-e^{2 i 3^{p+1} m},-k,1-\frac{1}{2} i 3^{-p-1} \log (a)\right)\right)}{k!}\\
=\frac{1}{2\pi i}\int_{C}\sum_{p=0}^{n-1}3^{-p} a^ww^{-k-1} \tan \left(3^{p+1} (m+w)\right)dw
\end{multline}
\subsubsection{Right-hand side first contour integral}
Use equation (\ref{gen_tan}) and replace $b$ by $3^{n}$ and multiply both sides by $3/2$ and simplify to get;
\begin{multline}\label{rhc_ss}
\frac{3 i \left(\log ^k(a)-2^{k+1} \left(i 3^n\right)^k e^{2 i m 3^n} \Phi \left(-e^{2 i 3^n m},-k,1-\frac{1}{2} i 3^{-n} \log (a)\right)\right)}{2 k!}\\
=\frac{1}{2\pi i}\int_{C}\frac{3}{2} a^w w^{-k-1} \tan
   \left(3^n (m+w)\right)dw
\end{multline}
\subsubsection{Right-hand side second contour integral}
Use equation (\ref{gen_tan}) and replace $b$ by $3^{n}$ and multiply both sides by $-\frac{3^{1-n}}{2}$ and simplify to get;
\begin{multline}
-\frac{i 3^{1-n} \left(\log ^k(a)-2^{k+1} \left(i 3^n\right)^k e^{2 i m 3^n} \Phi \left(-e^{2 i 3^n m},-k,1-\frac{1}{2} i 3^{-n} \log (a)\right)\right)}{2 k!}\\
=-\frac{1}{2\pi i}\int_{C}\frac{1}{2} 3^{1-n}a^w w^{-k-1} \tan \left(3^n (m+w)\right)dw
\end{multline}
\section{Derivation of the main theorems in terms of the Hurwitz-Lerch zeta function}
In this section we will derive and evaluate formulae involving the finite sum of the Hurwitz-Lerch zeta function in terms other special functions, trigonometric functions and fundamental constants.
\begin{theorem}
Main theorem: secant-sine series. For all $k,a,m \in\mathbb{C}$ then,
\begin{multline}\label{eq:theorem_ss}
\sum_{p=0}^{n-1}3^{-p} \left(\log ^k(a)+2^k \left(i 3^{p+1}\right)^k e^{2 i m 3^p} \left(-3 \Phi \left(-e^{2 i 3^{p+1}
   m},-k,\frac{1}{6} \left(2-i 3^{-p} \log (a)\right)\right)\right.\right. \\ \left.\left.
+3 e^{2 i m 3^p} \Phi \left(-e^{2 i 3^{p+1}
   m},-k,\frac{1}{6} \left(4-i 3^{-p} \log (a)\right)\right)\right.\right. \\ \left.\left.
   -2 e^{4 i m 3^p} \Phi \left(-e^{2 i 3^{p+1}
   m},-k,\frac{1}{6} \left(6-i 3^{-p} \log (a)\right)\right)\right)\right)\\
=\frac{1}{2} 3^{1-n}
   \left(\left(3^n-1\right) \log ^k(a)+2^{k+1} \left(\left(i 3^n\right)^k e^{2 i m 3^n} \Phi \left(-e^{2 i 3^n
   m},-k,1-\frac{1}{2} i 3^{-n} \log (a)\right)\right.\right. \\ \left.\left.-i^k e^{2 i m} 3^n \Phi \left(-e^{2 i m},-k,1-\frac{1}{2} i \log
   (a)\right)\right)\right)
\end{multline}
\end{theorem}
\begin{proof}
Observe that the addition of the right-hand sides of equations (\ref{lfci}) and (\ref{lsci}), is equal to the addition of the right-hand sides of equations (\ref{rfci}) and (\ref{rsci}) so we may equate the left-hand sides and simplify relative to equation (\ref{eq:contour}) and the Gamma function to yield the stated result.
\end{proof}
\begin{theorem}
Main theorem: cosecant-cosine series. For all $k,a,m \in\mathbb{C}$ then,
\begin{multline}\label{eq:theorem_cc}
\sum_{p=0}^{n-1}\left(i 3^{p+1}\right)^k e^{i m 3^p} \left(\Phi \left(e^{2 i 3^{p+1} m},-k,\frac{1}{6} \left(1-i 3^{-p} \log (a)\right)\right)\right. \\ \left.
+e^{4 i m 3^p} \Phi \left(e^{2 i 3^{p+1}
   m},-k,\frac{1}{6} \left(5-i 3^{-p} \log (a)\right)\right)\right)\\
=i^k e^{i m} \Phi \left(e^{2 i m},-k,\frac{1}{2}-\frac{1}{2} i \log (a)\right)-\left(i 3^n\right)^k e^{i m 3^n} \Phi
   \left(e^{2 i 3^n m},-k,\frac{1}{2}-\frac{1}{2} i 3^{-n} \log (a)\right)
\end{multline}
\end{theorem}
\begin{proof}
Observe that the addition of the right-hand side of equation (\ref{lfci_cc}), is equal to the addition of the right-hand sides of equations (\ref{rfci_cc}) and (\ref{rsci_cc}) so we may equate the left-hand sides and simplify relative to equation (\ref{eq:contour}) and the Gamma function to yield the stated result.
\end{proof}
\begin{theorem}
Main theorem: secant-sine series. For all $k,a,m \in\mathbb{C}$ then,
\begin{multline}\label{eq:theorem_ss1}
\sum_{p=0}^{n-1}3^{-p} \left(i 3^{p+1}\right)^k e^{2 i m 3^p} \left(3 \left(3^p-1\right) \Phi \left(-e^{2 i 3^{p+1} m},-k,\frac{1}{6} \left(3^{-p} a+2\right)\right)\right. \\ \left.
-3 \left(3^p-1\right) e^{2 i m 3^p} \Phi \left(-e^{2 i 3^{p+1}m},-k,\frac{1}{6} \left(3^{-p} a+4\right)\right)\right. \\ \left.
-2 e^{4 i m 3^p} \Phi \left(-e^{2 i 3^{p+1} m},-k,\frac{1}{6}\left(3^{-p} a+6\right)\right)\right)\\
=-3 i \left(3^n-1\right) \left(i 3^n\right)^{k-1} e^{2 i m 3^n} \Phi
   \left(-e^{2 i 3^n m},-k,\frac{3^{-n} a}{2}+1\right)
\end{multline}
\end{theorem}
\begin{proof}
Observe that the addition of the right-hand side of equations (\ref{lhfc_ss}), (\ref{lhsc_ss}) and (\ref{lhtc_ss}) are equal to the right-hand sides of equation (\ref{rhc_ss})  so we may equate the left-hand sides replace $a$ by $e^{ai}$ and simplify  relative to equation (\ref{eq:contour}) and the Gamma function to yield the stated result.
\end{proof}
\begin{example}
The degenerate case.
\begin{equation}
\sum_{p=0}^{n-1}3^{-p} \sin ^3\left(m 3^p\right) \sec \left(m 3^{p+1}\right)=\frac{3}{8} \left(3^{-n} \tan \left(m 3^n\right)-\tan (m)\right)
\end{equation}
\end{example}
\begin{proof}
Use equation (\ref{eq:theorem_ss}) and set $k=0$ and simplify using entry (2) in the Table below equation (64:12:7) on page 692 in \cite{atlas} to yield the stated result.
\end{proof}
\begin{example}
The degenerate case.
\begin{equation}
\sum_{p=0}^{n-1}\cos \left(2 m 3^p\right) \csc \left(m 3^{p+1}\right)=\frac{1}{2} \left(\csc (m)-\csc \left(m 3^n\right)\right)
\end{equation}
\end{example}
\begin{proof}
Use equation (\ref{eq:theorem_cc}) and set $k=0$ and simplify using entry (2) in the Table below equation (64:12:7) on page 692 in \cite{atlas} to yield the stated result.
\end{proof}
\begin{example}
The degenerate case.
\begin{multline}
\sum_{p=0}^{n-1}\frac{3^{-p} \left(-2 i \sin \left(2 m 3^p\right)-2 \cos \left(2 m 3^p\right)-i \left(3^{p+1}-4\right) \tan \left(m 3^p\right)+1\right)}{2 \cos \left(2 m 3^p\right)-1}\\
=\frac{3}{2}
   \left(3^{-n}-1\right) \left(1+i \tan \left(m 3^n\right)\right)
\end{multline}
\end{example}
\begin{proof}
Use equation (\ref{eq:theorem_ss1}) and set $k=0$ and simplify using entry (2) in the Table below equation (64:12:7) on page 692 in \cite{atlas} to yield the stated result.
\end{proof}
\section{Table of results: Part I}
In this section we evaluate Theorem (\ref{eq:theorem_ss}) for various values of the parameters involved to derive related closed form formulae in terms of special functions and fundamental constants.
\begin{example}
Finite product of quotient gamma functions. 
\begin{multline}\label{eq:gamma_ss}
\prod_{p=0}^{n-1}\left(\frac{\Gamma \left(3^{-p-1} a+\frac{1}{2}\right)}{\Gamma \left(3^{-p-1} a+1\right)}\right)^{2\times 3^{-p-1}}\\
 \left(\frac{9^{p+1} \Gamma \left(3^{-p-1} a+\frac{1}{6}\right) \Gamma \left(3^{-p-1}
   a+\frac{5}{6}\right)}{\left(3^p-a\right) \left(2\times 3^p-a\right) \Gamma \left(\frac{1}{3} \left(3^{-p} a-2\right)\right) \Gamma \left(\frac{1}{3} \left(3^{-p}
   a-1\right)\right)}\right)^{3^{-p}}\\
=\frac{3^{\frac{3}{4} \left(1-3^{-n}\right)} \Gamma \left(a+\frac{1}{2}\right) \left(\frac{\Gamma \left(3^{-n} a+1\right)}{\Gamma \left(3^{-n}
   a+\frac{1}{2}\right)}\right)^{3^{-n}}}{\Gamma (a+1)}
\end{multline}
\end{example}
\begin{proof}
Use equation (\ref{eq:theorem_ss}) and set $m=0$ and simplify in terms of the Hurwitz zeta function using entry (4) in the Table below equation (64:12:7) on page 692 in \cite{atlas}. Next take the first partial derivative with respect to $k$ and set $k=0$ and simplify in terms of the log-gamma function using equation (64:10:2) in \cite{atlas}. Finally take the exponential function of both sides and simplify both sides to yield the stated result.
\end{proof}
\begin{example}
Infinite product of quotient gamma functions. 
\begin{multline}
\prod_{p=0}^{\infty}\left(\frac{\Gamma \left(3^{-p-1} a+\frac{1}{2}\right)}{\Gamma \left(3^{-p-1} a+1\right)}\right)^{2\times 3^{-p-1}}\\
 \left(\frac{9^{p+1} \Gamma \left(3^{-p-1} a+\frac{1}{6}\right) \Gamma \left(3^{-p-1}
   a+\frac{5}{6}\right)}{\left(3^p-a\right) \left(2\times 3^p-a\right) \Gamma \left(\frac{1}{3} \left(3^{-p} a-2\right)\right) \Gamma \left(\frac{1}{3} \left(3^{-p}
   a-1\right)\right)}\right)^{3^{-p}}\\
=\frac{3^{3/4} \Gamma \left(a+\frac{1}{2}\right)}{\Gamma (a+1)}
\end{multline}
\end{example}
\begin{proof}
Use equation (\ref{eq:gamma_ss}) and take the limit as $n\to \infty$ of the right-hand side and simplify to yield the stated result.
\end{proof}
\begin{example}
Finite product of root of quotient cosine functions.
\begin{multline}\label{eq:cosine}
\prod_{p=0}^{n-1} \left(\frac{\left(\frac{\cos \left(3^p m\right)}{\cos \left(3^p r\right)}\right)^{16} \left(1-2 \cos \left(2\times 3^p r\right)\right)^2}{\left(1-2 \cos \left(2\times 3^p m\right)\right)^2}\right)^{3^{-2
   p}}=\left(\frac{\cos (m)}{\cos (r)}\right)^{18} \left(\frac{\cos \left(3^n r\right)}{\cos \left(3^n m\right)}\right)^{2\times 3^{2-2 n}}
\end{multline}
\end{example}
\begin{proof}
Use equation (\ref{eq:theorem_ss}) and form a second equation by replacing $m\to r$. Next take the differnce of the two equations and simplify. Then set $k=-1,a=1$ and simplify in terms of the logarithm function using entry (5) in the Table below equation (64:12:7) on page 692 in \cite{atlas}. Next take the exponential function of both sides and simplify to yield the stated result. 
\end{proof}
\begin{example}
Infinite product of root of quotient cosine functions. 
\begin{equation}\label{eq:qcf}
\prod_{p=0}^{\infty}\left(\frac{\left(\frac{\cos \left(3^p m\right)}{\cos \left(3^p
   r\right)}\right)^{16} \left(1-2 \cos \left(2\times 3^p
   r\right)\right)^2}{\left(1-2 \cos \left(2\times 3^p
   m\right)\right)^2}\right)^{\frac{3^{-2 p}}{18}}=\frac{\cos (m)}{\cos
   (r)}
\end{equation}
\end{example}
\begin{proof}
Use equation (\ref{eq:cosine}) and analyze the right-hand side as $n\to \infty$ using Figure 2.
\end{proof}
\begin{figure}[H]
\includegraphics[scale=0.8]{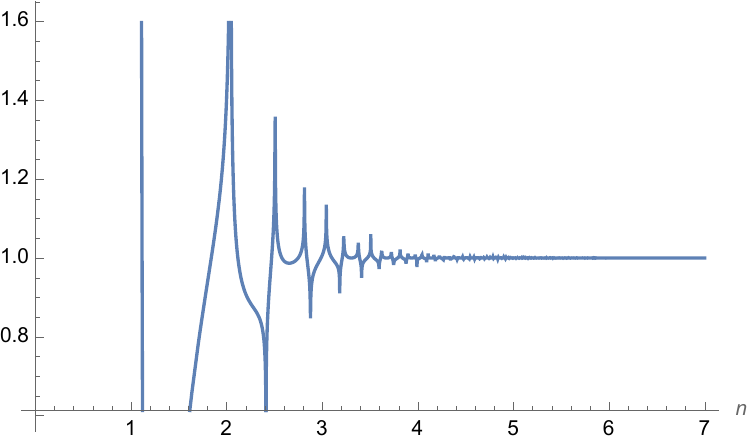}
\caption{Plot of  $f(m,r,n)=\left(\sec \left(m 3^n\right) \cos \left(3^n r\right)\right)^{2\times 3^{2-2 n}}$, $m,r\in\mathbb{R}$.}
   \label{fig:fig2}
\end{figure}
\vspace{-6pt}
\section{Quotient cosine functions with reciprocal angles: the Secant-sine series}
\begin{figure}[H]
\includegraphics[scale=0.8]{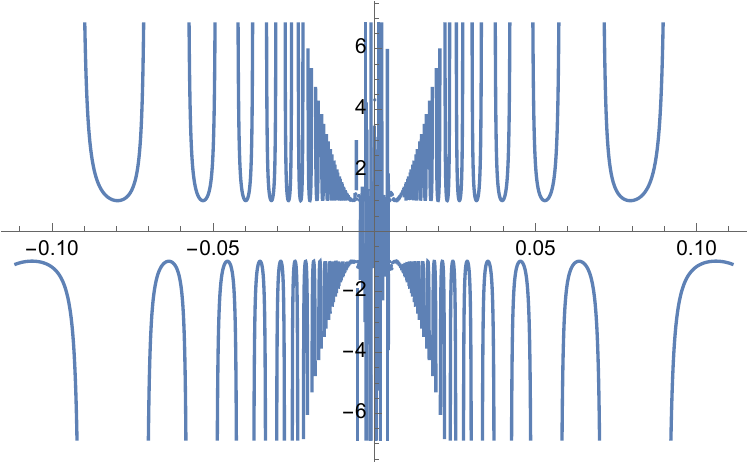}
\caption{Plot of  $f(m)=\cos (m) \sec \left(\frac{1}{m}\right)$, $m\in\mathbb{R}$.}
   \label{fig:fig2}
\end{figure}
\vspace{-6pt}
\begin{figure}[H]
\includegraphics[scale=0.8]{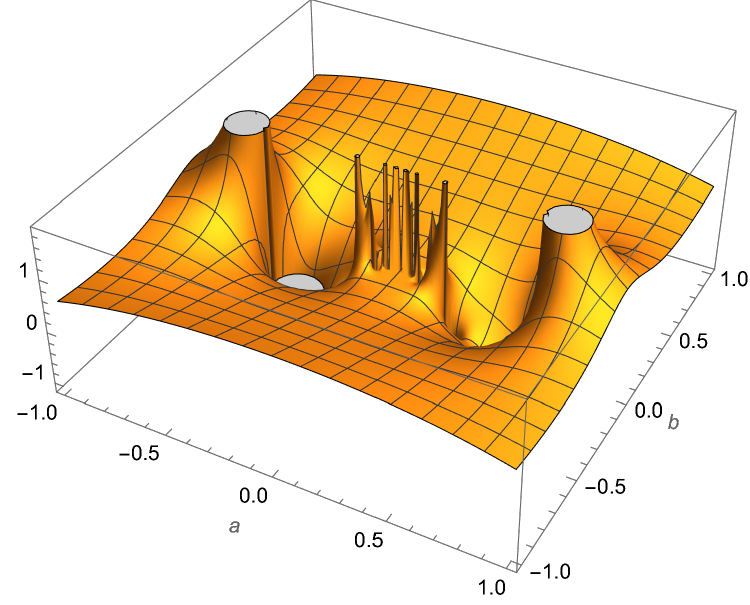}
\caption{Plot of  $f(m)=Re\left(\cos (m) \sec \left(\frac{1}{m}\right)\right)$, $m\in\mathbb{C}$.}
   \label{fig:fig2}
\end{figure}
\vspace{-6pt}
\begin{figure}[H]
\includegraphics[scale=0.8]{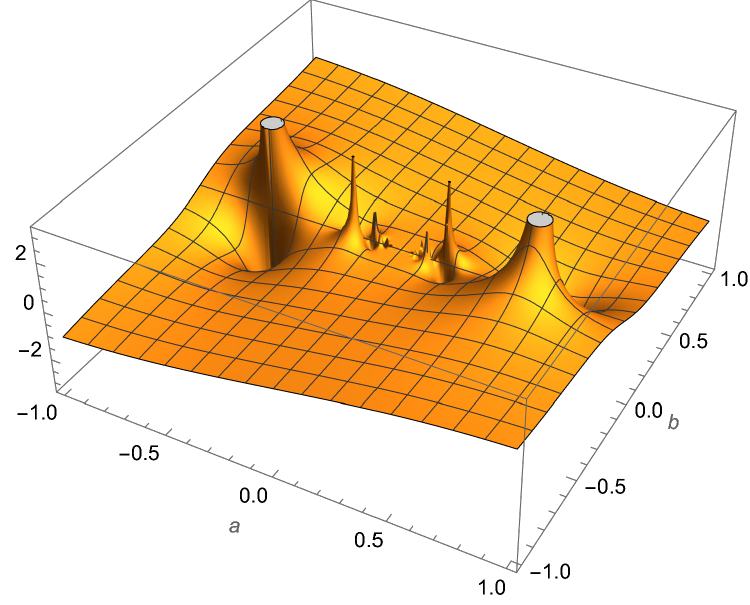}
\caption{Plot of  $f(m)=Im\left(\cos (m) \sec \left(\frac{1}{m}\right)\right)$, $m\in\mathbb{C}$.}
   \label{fig:fig2}
\end{figure}
\vspace{-6pt}
\begin{figure}[H]
\includegraphics[scale=0.8]{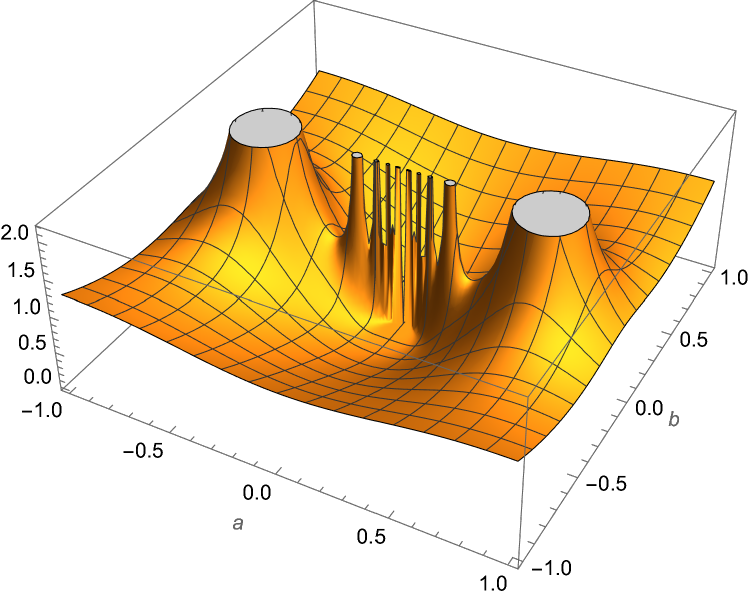}
\caption{Plot of  $f(m)=Abs\left(\cos (m) \sec \left(\frac{1}{m}\right)\right)$, $m\in\mathbb{C}$.}
   \label{fig:fig2}
\end{figure}
\vspace{-6pt}
\subsection{Elliptic functions}
A traditional area of mathematics called elliptic functions \cite{lawden} is used often in physics and engineering applications, either directly or as a building block for more complex ones. However, it cannot be denied that they are not as well recognised as they ought to be. It is typically feasible to discover an even simpler (and not always sufficient) skeleton of the theory in trigonometry in many instances where the elliptic functions might be usefully applied, or at least examined as a potentially effective first approximation. Here we present a simple infinite product in terms of the ratio involving elliptic functions. In this example we look at an extension of the ratio of Elliptic functions in terms of the infinite product involving the ratio of the cosine function using equation (\ref{eq:qcf}) and equation (29') in \cite{zaroodny}, where $-1< Re(k) < Re(a) < Re(b)<1$.
\begin{example}
Infinite product of sine and cosine functions in terms of Elliptic functions. 
\begin{multline}
\prod_{p=0}^{\infty}\left(\frac{\cos ^{16}\left(3^p k \sin (a)\right) \left(1-2 \cos \left(2\times 3^p k \sin (b)\right)\right)^2 \sec ^{16}\left(3^p k \sin (b)\right)}{\left(1-2 \cos
   \left(2\times 3^p k \sin (a)\right)\right)^2}\right)^{\frac{1}{2} 3^{-2-2 p}}\\
=\frac{\cos (k \sin (a))}{\cos (k \sin (b))}=\frac{\sqrt{1-(k \sin (a))^2}}{\sqrt{1-(k \sin
   (b))^2}}
\end{multline}
\end{example}
\begin{example}
Functional equation in terms of the Hurwitz-Lerch zeta function. 
\begin{multline}
\Phi (z,s,a)
=3^{-2 s-1} \left(3^s \left(3 \Phi \left(z^3,s,\frac{a}{3}\right)+z \left(3 \Phi \left(z^3,s,\frac{a+1}{3}\right)+2 z \Phi \left(z^3,s,\frac{a+2}{3}\right)\right)\right)\right. \\ \left.+z^2 \left(\Phi
   \left(z^9,s,\frac{a+2}{9}\right)+z^6 \Phi \left(z^9,s,\frac{a+8}{9}\right)+z^3 \Phi \left(z^9,s,\frac{a+5}{9}\right)\right)\right)
\end{multline}
\end{example}
\begin{proof}
Use equation (\ref{eq:theorem_ss}) and set $n=2,m=\frac{\log(-z)}{2i},a=e^{ai},k=-s$ and simplify.
\end{proof}
\section{Special cases of the Hurwitz-Lerch zeta function}
In this section we will evaluate equation (\ref{eq:theorem_ss}) and derive formulae in terms of fundamental constants; namely Catalan's constant $C$ given in equations (20) and (21) in \cite{guillera}, Glaisher's constant $A$, given in equation (18) in \cite{guillera}, Apery's constant $\zeta(3)$, given in equation (19).
\begin{example}
Finite series in terms of Catalan's constant $C$. 
\begin{multline}
\sum_{p=0}^{n-1}27^{-p} \left(3 \psi ^{(1)}\left(\frac{1}{12} \left(2+3^{-p}\right)\right)-3 \psi ^{(1)}\left(\frac{1}{12} \left(4+3^{-p}\right)\right)+2 \psi ^{(1)}\left(\frac{1}{12}
   \left(6+3^{-p}\right)\right)\right. \\ \left.
-3 \psi ^{(1)}\left(\frac{1}{12} \left(8+3^{-p}\right)\right)+3 \psi ^{(1)}\left(\frac{1}{12} \left(10+3^{-p}\right)\right)-2 \left(8\times 9^{p+1}+\psi
   ^{(1)}\left(\frac{1}{12} \left(12+3^{-p}\right)\right)\right)\right)\\
=27^{1-n} \left(8\times 9^n \left(-2 C 3^n+3^n+1\right)+\psi ^{(1)}\left(1+\frac{3^{-n}}{4}\right)-\psi
   ^{(1)}\left(\frac{1}{4} \left(2+3^{-n}\right)\right)\right)
\end{multline}
\end{example}
\begin{proof}
Use equation (\ref{eq:theorem_ss}) and set $m=0,k=-2,a=e^{i}$ and simplify.
\end{proof}
\begin{example}
Difference of Hurwitz-Lerch zeta functions in terms of Glaisher's constant $A$.
\begin{equation}
\Phi'\left(-1,-1,\frac{1}{3}\right)-\Phi'\left(-1,-1,\frac{2}{3}\right)=\log \left(\frac{2^{2/9} \sqrt[12]{3}
   \sqrt[6]{e}}{A^2}\right)
\end{equation}
\end{example}
\begin{proof}
Use equation (\ref{eq:theorem_ss}) take the first partial derivative with respect to $k$ and set $m=0,k=a=1$ and simplify.
\end{proof}
\begin{example}
Difference of Hurwitz-Lerch zeta functions in terms of Apery's constant $\zeta](3)$.
\begin{equation}
\Phi'\left(-1,-2,\frac{2}{3}\right)-\Phi'\left(-1,-2,\frac{1}{3}\right)=\frac{14 \zeta (3)}{9 \pi ^2}
\end{equation}
\end{example}
\begin{proof}
Use equation (\ref{eq:theorem_ss}) take the first partial derivative with respect to $k$ and set $m=0,k=2,a=1$ and simplify.
\end{proof}
\begin{example}
Finite series in terms of Catalan's constant $C$.
\begin{multline}
\sum_{p=0}^{n-1}\left(18 \Phi'\left(-1,-1,\frac{1}{6} \left(2-3^{-p}\right)\right)-18 \Phi'\left(-1,-1,\frac{1}{6} \left(4-3^{-p}\right)\right)\right. \\ \left.
+12
   \Phi'\left(-1,-1,\frac{1}{6} \left(6-3^{-p}\right)\right)+3^{-p} \log \left(i 3^{p+1}\right)\right)\\
=\frac{-6 \pi 
   \Phi'\left(-1,-1,1-\frac{3^{-n}}{2}\right)+6 C+\frac{3}{2} \left(\pi -\pi  3^{-n}\right) \log \left(i 3^n\right)}{\pi }
\end{multline}
\end{example}
\begin{proof}
Use equation (\ref{eq:theorem_ss}) take the first partial derivative with respect to $k$ and set $m=0,k=1,a=e^{-i}$ and simplify.
\end{proof}
\begin{example}
Finite product of quotient gamma functions. 
\begin{multline}\label{eq:q_gamma}
\prod_{p=0}^{n-1}\left(\frac{3^{a 3^{-p}+2 p} \Gamma \left(3^{-p-1} a-1\right)}{\left(a^2-a
   3^{p+1}+2\times 9^p\right) \Gamma \left(3^{-p-1} a+1\right)^{2/3} \Gamma
   \left(3^{-p} a-3\right)}\right)^{3^{-p}}\\
=\frac{3^{\frac{1}{8} 9^{-n} \left(9 a
   \left(9^n-1\right)+8\ 3^n \left(-n+3^{n+1}-3\right)\right)} a^{3^{-n}} \Gamma
   \left(3^{-n} a\right)^{3^{-n}}}{\Gamma (a+1)}
\end{multline}
\end{example}
\begin{proof}
Use equation (\ref{eq:theorem_ss}) and set $m=\pi/2$ and simplify in terms of the Hurwitz zeta function using entry (4) in the Table below equation (64:12:7) on page 692 in \cite{atlas}. Next take the first partial derivative with respect to $k$ and set $k=0$ and simplify in terms of the log-gamma function using equation (64:10:2) in \cite{atlas}. Finally take the exponential function of both sides and simplify both sides to yield the stated result.
\end{proof}
\begin{example}
Infinite product of quotient gamma functions. 
\begin{equation}
\prod_{p=0}^{\infty}\left(\frac{3^{a 3^{-p}+2 p} \Gamma \left(3^{-p-1} a-1\right)}{\left(a^2-a
   3^{p+1}+2\times 9^p\right) \Gamma \left(3^{-p-1} a+1\right)^{2/3} \Gamma
   \left(3^{-p} a-3\right)}\right)^{3^{-p}}=\frac{3^{\frac{9 a}{8}+3}}{\Gamma
   (a+1)}
\end{equation}
\end{example}
\begin{proof}
Use equation (\ref{eq:q_gamma}) and take the limit of the right-hand side as $n\to \infty$ and simplify.
\end{proof}
\begin{example}
Finite product involving the cosine function. 
\begin{equation}\label{eq:cosine_jq}
\prod_{p=0}^{n-1}\left(\frac{\left(e^{-2 m 3^{p+1}}+1\right)^{2/3} \cosh ^2\left(m 3^p\right)}{2 \cosh \left(2 m 3^p\right)-1}\right)^{3^{-2
   p}}
   =\frac{2^{\frac{3}{4} \left(1+3^{1-2 n}\right)} \cosh ^3(m)}{e^{3 m} \left(1+e^{-2\times 3^n
   m}\right)^{3^{1-2 n}}}
\end{equation}
\end{example}
\begin{proof}
Use equation (\ref{eq:theorem_ss}) and set $a=1$ and simplify in terms of the Polylogarithm function using equation (64:12:2) in \cite{atlas}. Next simplify the Polylogarithm function in terms of the Hurwitz zeta function using equation (6) in \cite{jonq}. Next take the limit of both sides as $k\to -1$ and simplify in terms of the log-gamma function using equation (64:10:2) in \cite{atlas}. Next take the exponential function of both sides and simplify.
\end{proof}
\begin{example}
Infinite product involving the cosine function. 
\begin{equation}
\prod_{p=0}^{\infty}\left(\frac{\left(e^{-2 m 3^{p+1}}+1\right)^{2/3} \cosh ^2\left(m 3^p\right)}{2 \cosh \left(2 m 3^p\right)-1}\right)^{3^{-2
   p-1}}=\frac{e^{-2 m}+1}{2^{3/4}}
\end{equation}
\end{example}
\begin{proof}
Use equation (\ref{eq:cosine_jq}) and take the limit as $n\to \infty$ and simplify.
\end{proof}
%
%

\section{Table of results: Part II}
In this section we evaluate Theorem (\ref{eq:theorem_cc}) for various values of the parameters involved to derive related closed form formulae in terms of special functions and fundamental constants.
\begin{example}
Finite product of gamma functions. 
\begin{equation}\label{eq:gamma_cc}
\prod_{p=0}^{n-1}\frac{\Gamma \left(\frac{1}{6} \left(3^{-p} a+1\right)\right) \Gamma \left(\frac{1}{6} \left(3^{-p}
   a+5\right)\right)}{2 \pi }=\frac{3^{-\frac{1}{4} a 3^{1-n} \left(3^n-1\right)} \Gamma
   \left(\frac{a+1}{2}\right)}{\Gamma \left(\frac{1}{2} \left(3^{-n} a+1\right)\right)}
\end{equation}
\end{example}
\begin{proof}
Use equation (\ref{eq:theorem_cc}) and set $m=0$ and simplify in terms of the Hurwitz zeta function using entry (4) in the Table below equation (64:12:7) on page 692 in \cite{atlas}. Next take the first partial derivative with respect to $k$ and set $k=0$ and simplify in terms of the log-gamma function using equation (64:10:2) in \cite{atlas}. Finally take the exponential function of both sides and simplify both sides to yield the stated result.
\end{proof}
\begin{example}
Infinite product of gamma functions. 
\begin{equation}
\prod_{p=0}^{\infty}\frac{\Gamma \left(\frac{1}{6} \left(3^{-p} a+1\right)\right) \Gamma \left(\frac{1}{6} \left(3^{-p}
   a+5\right)\right)}{2 \pi }=\frac{3^{-3 a/4} \Gamma \left(\frac{a+1}{2}\right)}{\sqrt{\pi }}
\end{equation}
\end{example}
\begin{proof}
Use equation (\ref{eq:gamma_cc}) and take the limit as $n\to \infty$ and simplify.
\end{proof}
\begin{example}
Finite product of quotient hyperbolic tangent functions. 
\begin{multline}\label{eq:q_hyper_tan}
\prod_{p=0}^{n-1}\left(\frac{\left(1+2 \cosh \left(2\times 3^p m\right)\right) \left(-1+2 \cosh \left(2\times 3^p r\right)\right)}{\left(-1+2 \cosh \left(2\times 3^p m\right)\right) \left(1+2 \cosh \left(2\times 3^p
   r\right)\right)}\right)^{3^{-p}} \left(\frac{\tanh \left(3^p r\right)}{\tanh \left(3^p m\right)}\right)^{2\times 3^{-p}}\\
=\left(\frac{\tanh \left(3^n m\right)}{\tanh \left(3^n
   r\right)}\right)^{3^{1-n}} \left(\frac{\tanh (r)}{\tanh (m)}\right)^3
\end{multline}
\end{example}
\begin{proof}
Use equation (\ref{eq:theorem_cc}) and form a second equation by replacing $m\to r$. Next take the differnce of the two equations and simplify. Then set $k=-1,a=1$ and simplify in terms of the logarithm function using entry (5) in the Table below equation (64:12:7) on page 692 in \cite{atlas}. Next take the exponential function of both sides and simplify to yield the stated result. 
\end{proof}
\begin{example}
Infinite product of quotient hyperbolic tangent functions. 
\begin{multline}
\prod_{p=0}^{\infty}\left(\frac{\left(1+2 \cosh \left(2\times 3^p m\right)\right) \left(-1+2 \cosh \left(2\times 3^p r\right)\right)}{\left(-1+2 \cosh \left(2\times 3^p m\right)\right) \left(1+2 \cosh \left(2\times 3^p
   r\right)\right)}\right)^{3^{-p}} \left(\frac{\tanh \left(3^p r\right)}{\tanh \left(3^p m\right)}\right)^{2\times 3^{-p}}\\
=\left(\frac{\tanh (r)}{\tanh (m)}\right)^3
\end{multline}
\end{example}
\begin{proof}
Use equation (\ref{eq:q_hyper_tan}) and take the limit as $n\to \infty$ of the right-hand side and simplify.
\end{proof}
\section{Quotient cosine functions with reciprocal angles: the cosecant-cosine series}
In this section we look at plots involving the ratio of the hyperbolic tangent function and reciprocal angles.
\begin{figure}[H]
\includegraphics[scale=0.8]{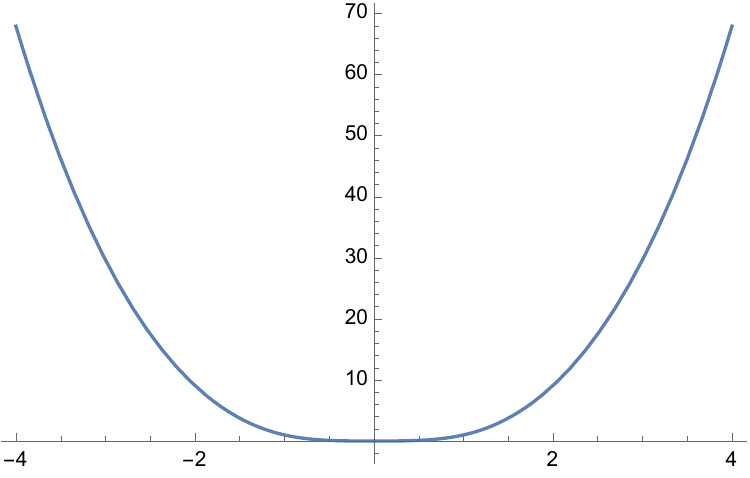}
\caption{Plot of  $f(m)=\tanh ^3(r) \coth ^3\left(\frac{1}{r}\right)$, $m\in\mathbb{R}$.}
   \label{fig:fig2}
\end{figure}
\vspace{-6pt}
\begin{figure}[H]
\includegraphics[scale=0.8]{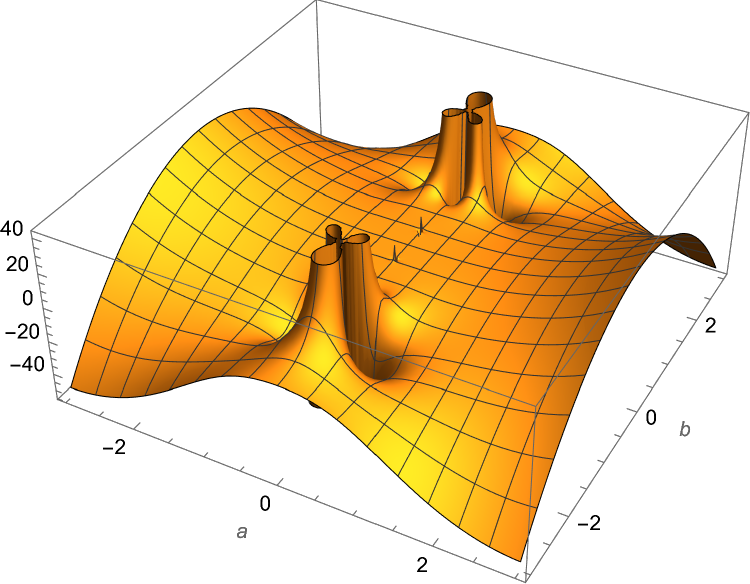}
\caption{Plot of  $f(r)=Re\left(\tanh ^3(r) \coth ^3\left(\frac{1}{r}\right)\right)$, $r\in\mathbb{C}$.}
   \label{fig:fig2}
\end{figure}
\vspace{-6pt}
\begin{figure}[H]
\includegraphics[scale=0.8]{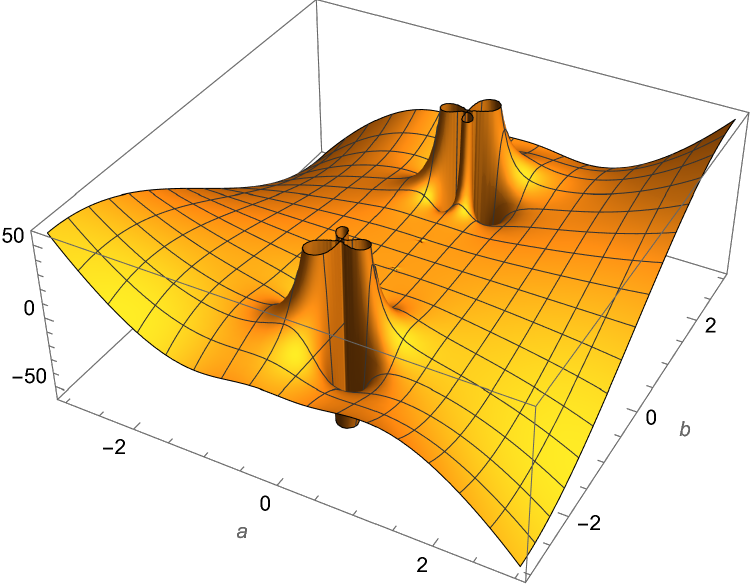}
\caption{Plot of  $f(r)=Im\left(\tanh ^3(r) \coth ^3\left(\frac{1}{r}\right)\right)$, $r\in\mathbb{C}$.}
   \label{fig:fig2}
\end{figure}
\vspace{-6pt}
\begin{figure}[H]
\includegraphics[scale=0.8]{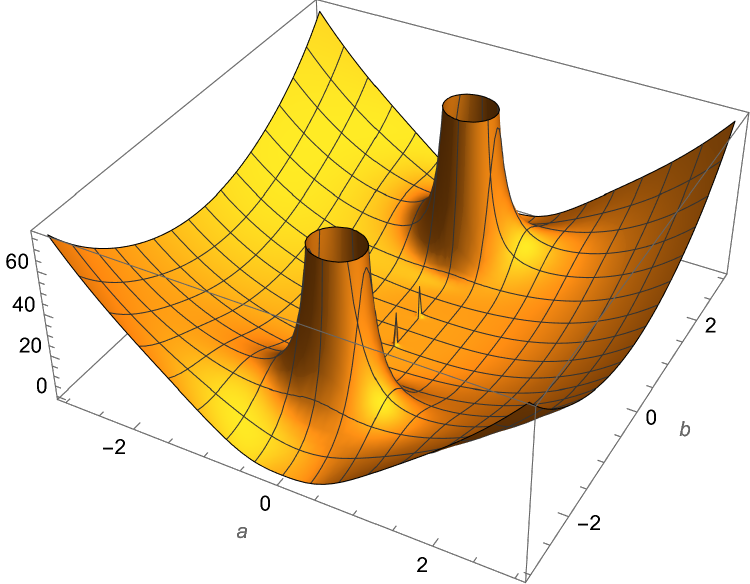}
\caption{Plot of  $f(r)=Abs\left(\tanh ^3(r) \coth ^3\left(\frac{1}{r}\right)\right)$, $r\in\mathbb{C}$.}
   \label{fig:fig2}
\end{figure}
\vspace{-6pt}
\begin{example}
Functional equation for the Hurwitz-Lerch zeta function.
\begin{multline}
\Phi (z,s,a)=3^{-s} \left(\Phi \left(z^3,s,\frac{a}{3}\right)+z \left(\Phi \left(z^3,s,\frac{a+1}{3}\right)+z \Phi \left(z^3,s,\frac{a+2}{3}\right)\right)\right)
\end{multline}
\end{example}
\begin{proof}
Use equation (\ref{eq:theorem_cc}) and set $n=1,m=\frac{\log(z)}{2i},a=e^{ai},k=-s$ and simplify.
\end{proof}
\begin{example}
Finite series in terms of Catalan's constant C.
\begin{multline}
\sum_{p=0}^{n-1}9^{-p} \left(\psi ^{(1)}\left(\frac{1}{12} \left(2+3^{-p}\right)\right)+\psi ^{(1)}\left(\frac{1}{12} \left(10+3^{-p}\right)\right)\right)\\
=9 \left(-8 C-9^{-n} \psi
   ^{(1)}\left(\frac{1}{4} \left(2+3^{-n}\right)\right)+\pi ^2\right)
\end{multline}
\end{example}
\begin{proof}
Use equation (\ref{eq:theorem_cc}) and take the first partial derivative with respect to $k$ and set $k=-2,m=0,a=e^{i/2}$ and simplify in terms of the Polygamma function using equation (64:12:2) in \cite{atlas}.
\end{proof}
\begin{example}
Finite series in terms of Catalan's constant C. 
\begin{multline}
\sum_{p=0}^{n-1}9^{-p} \left(\Phi'\left(1,2,\frac{1}{12} \left(2-3^{-p}\right)\right)+\Phi'\left(1,2,\frac{1}{6}
   \left(5-\frac{3^{-p}}{2}\right)\right)\right. \\ \left.
-\log \left(i 3^{p+1}\right) \left(\psi ^{(1)}\left(\frac{1}{12} \left(2-3^{-p}\right)\right)+\psi ^{(1)}\left(\frac{1}{12}
   \left(10-3^{-p}\right)\right)\right)\right)\\
=9 \left(-9^{-n}
   \Phi'\left(1,2,\frac{1}{2}-\frac{3^{-n}}{4}\right)+\Phi'\left(1,2,\frac{1}{4}\right)-\frac{1}{2} i \pi  \left(8 C+\pi
   ^2\right)\right. \\ \left.
+9^{-n} \log \left(i 3^n\right) \psi ^{(1)}\left(\frac{1}{2}-\frac{3^{-n}}{4}\right)\right)
\end{multline}
\end{example}
\begin{proof}
Use equation (\ref{eq:theorem_cc}) and take the first partial derivative with respect to $k$ and set $k=-2,m=0,a=e^{-i/2}$ and simplify in terms of the Polygamma function using equation (64:12:2) in \cite{atlas}.

\end{proof}
\begin{example}
Finite product of quotient gamma functions.
\begin{multline}\label{eq:q_gamma_cc1}
\prod_{p=0}^{n-1}\left(\frac{\Gamma \left(\frac{1}{12} \left(3^{-p} a+7\right)\right) \Gamma \left(\frac{1}{12} \left(3^{-p}
   a+11\right)\right)}{\Gamma \left(\frac{1}{4} 3^{-p-1} \left(a+3^p\right)\right) \Gamma \left(\frac{1}{12}
   \left(3^{-p} a+5\right)\right)}\right)^{(-1)^p}\\
=\frac{3^{\frac{1}{4} \left((-1)^n-1\right)} \Gamma
   \left(\frac{a+3}{4}\right) \left(\frac{\Gamma \left(\frac{1}{4} \left(3^{-n} a+1\right)\right)}{\Gamma
   \left(\frac{1}{4} \left(3^{-n} a+3\right)\right)}\right)^{(-1)^n}}{\Gamma \left(\frac{a+1}{4}\right)}
\end{multline}
\end{example}
\begin{proof}
Use equation (\ref{eq:theorem_cc}) and set $m=\pi/2$ and simplify in terms of the Hurwitz zeta function using entry (4) in the Table below equation (64:12:7) on page 692 in \cite{atlas}. Next take the first partial derivative with respect to $k$ and set $k=0$ and simplify in terms of the log-gamma function using equation (64:10:2) in \cite{atlas}. Finally take the exponential function of both sides and simplify both sides to yield the stated result.
\end{proof}
\begin{example}
An upper-bound for an infinite product of the ratio of gamma functions.
\begin{equation}
\prod_{p=0}^{\infty}\left(\frac{\Gamma \left(\frac{1}{12} \left(3^{-p} a+7\right)\right) \Gamma \left(\frac{1}{12} \left(3^{-p}
   a+11\right)\right)}{\Gamma \left(\frac{1}{4} 3^{-p-1} \left(a+3^p\right)\right) \Gamma \left(\frac{1}{12} \left(3^{-p}
   a+5\right)\right)}\right)^{(-1)^p}<\frac{\Gamma \left(\frac{1}{4}\right) \Gamma \left(\frac{a+3}{4}\right)}{\Gamma
   \left(\frac{3}{4}\right) \Gamma \left(\frac{a+1}{4}\right)}
\end{equation}
\end{example}
\begin{proof}
Use equation (\ref{eq:q_gamma_cc1}) and apply the limit as $n\to \infty$ to the right-hand side. In this analysis we see that the three exponential terms which have bounds as $n\to \infty$ namely $\frac{\Gamma \left(\frac{a+3}{4}\right) 3^{\frac{1}{4} \left(-1+e^{2 i \times[0,\pi]}\right)} 
 }{\Gamma \left(\frac{a+1}{4}\right)}\left(\frac{\Gamma\left(\frac{1}{4}\right)}{\Gamma \left(\frac{3}{4}\right)}\right)^{e^{2 i \times[0,\pi]}}$. We can see the bound on the right-hand side using Figure 11. over finite values of $n$.
\end{proof}
\begin{figure}[H]
\includegraphics[scale=0.8]{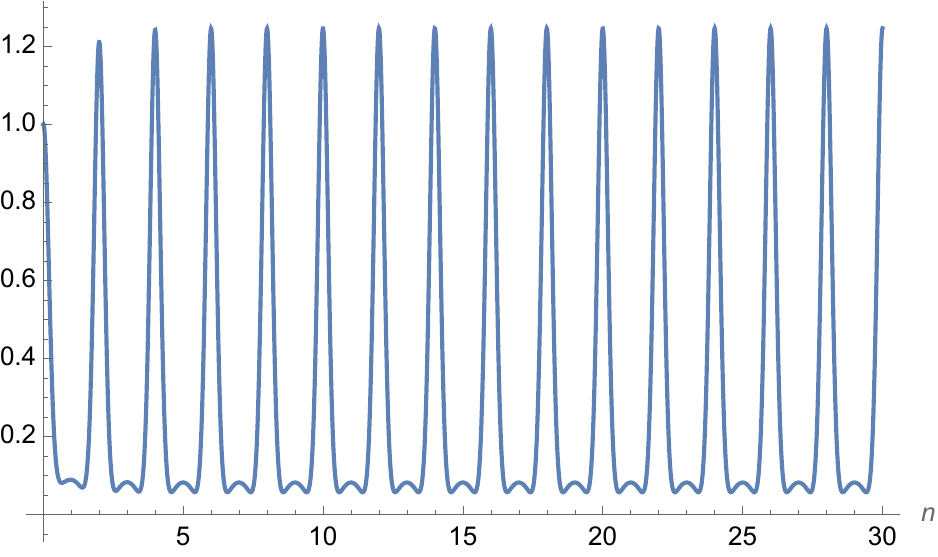}
\caption{Plot of  $\frac{3^{\frac{1}{4} \left((-1)^n-1\right)} \Gamma \left(\frac{a+3}{4}\right) \left(\frac{\Gamma \left(\frac{1}{4} \left(3^{-n}
   a+1\right)\right)}{\Gamma \left(\frac{1}{4} \left(3^{-n} a+3\right)\right)}\right)^{(-1)^n}}{\Gamma \left(\frac{a+1}{4}\right)}$, $a\in\mathbb{R}$.}
   \label{fig:fig2}
\end{figure}
%
%
\begin{example}
Finite product involving the cosine function. 
\begin{multline}\label{eq:coscos}
\prod_{p=0}^{n}\left(\frac{\left(\sqrt{3}+i \tanh \left(\frac{3^p x}{2}\right)\right) \left(\sqrt{3}-3 i \tanh \left(\frac{1}{2} 3^{p+1}
   x\right)\right)}{\left(\sqrt{3}-3 i \tanh \left(\frac{3^p x}{2}\right)\right) \left(\sqrt{3}+i \tanh \left(\frac{1}{2} 3^{p+1}x\right)\right)}\right)^{3^{-p-1}}\\
 \left(\tanh \left(\frac{1}{2} 3^{p+1} x\right) \coth \left(\frac{3^p x}{2}\right)\right)^{2\times3^{-p-1}} \\\left(\sinh \left(2\times 3^{-p-1}\left(\tanh ^{-1}\left(\frac{1}{2} \left(1-i \sqrt{3} \tanh \left(\frac{3^p
   x}{2}\right)\right)\right)\right.\right.\right. \\ \left.\left.\left.
-\tanh ^{-1}\left(\frac{1}{2} \left(1-i \sqrt{3} \tanh \left(\frac{1}{2} 3^{p+1}
   x\right)\right)\right)\right)\right)\right. \\ \left.
+\cosh \left(2\times 3^{-p-1} \left(\tanh ^{-1}\left(\frac{1}{2} \left(1-i \sqrt{3} \tanh\left(\frac{3^p x}{2}\right)\right)\right) \right.\right.\right. \\ \left.\left.\left.
-\tanh ^{-1}\left(\frac{1}{2} \left(1-i \sqrt{3} \tanh \left(\frac{1}{2} 3^{p+1}x\right)\right)\right)\right)\right)\right)\\
=\frac{\left(2 \cos \left(\frac{x}{3}\right)+1\right) \left(\tan
   \left(\frac{1}{2} 3^{n-1} x\right) \cot \left(\frac{3^n x}{2}\right)\right)^{3^{-n}}}{2 \cos
   \left(\frac{x}{3}\right)-1}
\end{multline}
\end{example}
\begin{proof}
Use equation (\ref{eq:theorem_cc}) and set $k=1,a=1,m=x$ and apply the method in section (8) in \cite{reyn_ejpam}.
\end{proof}
\begin{example}
Infinite product involving the cosine function. 
\begin{multline}
\prod_{p=0}^{\infty}\left(\frac{\left(\sqrt{3}+i \tanh \left(\frac{3^p x}{2}\right)\right) \left(\sqrt{3}-3 i \tanh \left(\frac{1}{2} 3^{p+1}
   x\right)\right)}{\left(\sqrt{3}-3 i \tanh \left(\frac{3^p x}{2}\right)\right) \left(\sqrt{3}+i \tanh \left(\frac{1}{2} 3^{p+1}x\right)\right)}\right)^{3^{-p-1}}\\
 \left(\tanh \left(\frac{1}{2} 3^{p+1} x\right) \coth \left(\frac{3^p x}{2}\right)\right)^{2\times3^{-p-1}} \\\left(\sinh \left(2\times 3^{-p-1}\left(\tanh ^{-1}\left(\frac{1}{2} \left(1-i \sqrt{3} \tanh \left(\frac{3^p
   x}{2}\right)\right)\right)\right.\right.\right. \\ \left.\left.\left.
-\tanh ^{-1}\left(\frac{1}{2} \left(1-i \sqrt{3} \tanh \left(\frac{1}{2} 3^{p+1}
   x\right)\right)\right)\right)\right)\right. \\ \left.
+\cosh \left(2\times 3^{-p-1} \left(\tanh ^{-1}\left(\frac{1}{2} \left(1-i \sqrt{3} \tanh\left(\frac{3^p x}{2}\right)\right)\right) \right.\right.\right. \\ \left.\left.\left.
-\tanh ^{-1}\left(\frac{1}{2} \left(1-i \sqrt{3} \tanh \left(\frac{1}{2} 3^{p+1}x\right)\right)\right)\right)\right)\right)\\
=\frac{2 \cosh (x)+1}{2 \cosh (x)-1}
\end{multline}
\end{example}
\begin{proof}
Use equation (\ref{eq:coscos}) set $x=3xi$ and analyze the right-hand side as $n\to \infty$ using Figure 12.
\end{proof}
\begin{figure}[H]
\includegraphics[scale=0.8]{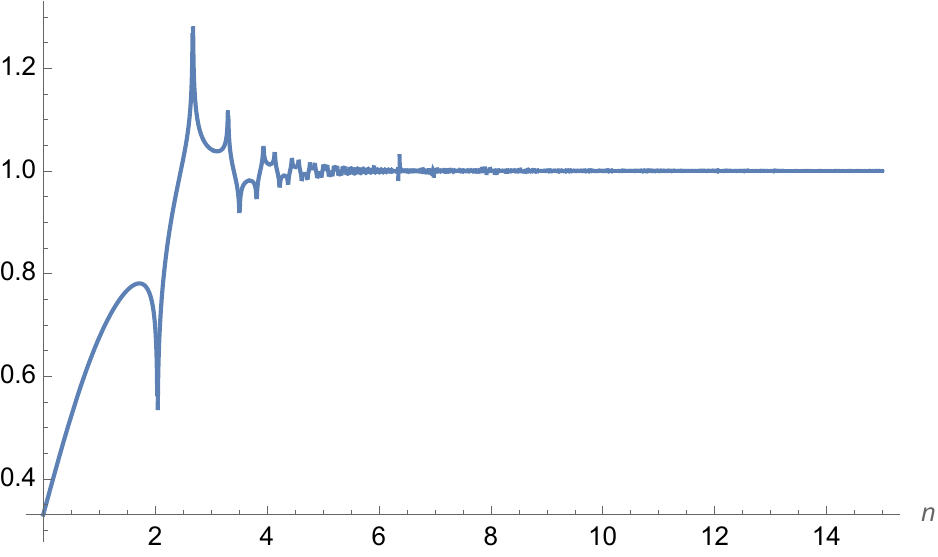}
\caption{Plot of  $\left(\tan \left(\frac{1}{2} 3^{n-1} x\right) \cot \left(\frac{3^n x}{2}\right)\right)^{3^{-n}}$, $x\in\mathbb{R}$.}
   \label{fig:fig3}
\end{figure}
\begin{example}
Finite product involving the Hurwitz-Lerch zeta function. 
\begin{multline}
\prod_{p=0}^{n-1}\exp \left(3^{-p} e^{-5 m 3^p} \left(e^{4 m 3^p} \Phi \left(e^{-2\times 3^{p+1} m},1,\frac{1}{6}\left(1+3^{-p}\right)\right)\right.\right. \\ \left.\left.
+\Phi \left(e^{-2\times 3^{p+1} m},1,\frac{1}{6}
   \left(5+3^{-p}\right)\right)\right)\right)\\
=2^{-3 e^m} (\coth (m)+1)^{3 e^m} \exp \left(-3^{1-n} e^{
   -m3^n} \Phi \left(e^{-2\times 3^n m},1,\frac{1}{2} \left(1+3^{-n}\right)\right)\right)
\end{multline}
\end{example}
\begin{proof}
Use equation (\ref{eq:theorem_cc}) and set $a=e^i$ and simplify in terms of the Polylogarithm function using equation (64:12:2) in \cite{atlas}. Next simplify the Polylogarithm function in terms of the Hurwitz zeta function using equation (6) in \cite{jonq}. Next take the limit of both sides as $k\to -1$ and simplify in terms of the log-gamma function using equation (64:10:2) in \cite{atlas}. Next take the exponential function of both sides and simplify.
\end{proof}
\begin{example}
Finite product involving the Hurwitz-Lerch zeta function.
\begin{multline}
\prod_{p=0}^{n-1}\exp \left(3^{-p} e^{i m 3^p} \left(\Phi \left(e^{2 i 3^{p+1} m},1,\frac{1}{6} \left(1+3^{n-p}\right)\right)
+e^{4 i m 3^p}\Phi \left(e^{2 i 3^{p+1} m},1,\frac{1}{6} \left(5+3^{n-p}\right)\right)\right)\right)\\
=\left(1-e^{2 i m 3^n}\right)^{3^{1-n}
   e^{-i m 3^n}} e^{3 e^{i m} \Phi \left(e^{2 i m},1,\frac{1}{2} \left(1+3^n\right)\right)}
\end{multline}
\end{example}
\begin{proof}
Use equation (\ref{eq:theorem_cc}) and set $a=e^{I 3^n}$ and simplify in terms of the Polylogarithm function using equation (64:12:2) in \cite{atlas}. Next simplify the Polylogarithm function in terms of the Hurwitz zeta function using equation (6) in \cite{jonq}. Next take the limit of both sides as $k\to -1$ and simplify in terms of the log-gamma function using equation (64:10:2) in \cite{atlas}. Next take the exponential function of both sides and simplify.
\end{proof}
\begin{example}
Finite series involving the Glaisher-Kinkelin constant $A$. 
\begin{multline}
\sum_{p=0}^{n-1}\frac{\csc ^2\left(\frac{\pi  3^p}{2}\right)}{8 \left(2 \cos \left(\pi 
   3^p\right)+1\right)^2} \left(2 \log \left(i 3^{p+1}\right) \left(3^p \left(5 \cos \left(\frac{\pi 
   3^p}{2}\right)+\cos \left(\frac{5 \pi  3^p}{2}\right)\right)\right.\right. \\ \left.\left.
-i \left(\sin \left(\frac{\pi  3^p}{2}\right)+\sin \left(\frac{5\pi  3^p}{2}\right)\right)\right)-2\times 3^{p+1} e^{-\frac{5}{2} i \pi  3^p} \left(-1+e^{i \pi  3^{p+1}}\right)^2 \right. \\ \left.
\left(\Phi'\left(-1,-1,\frac{1}{6} \left(3^{-p}+1\right)\right)+e^{2 i \pi  3^p}\Phi'\left(-1,-1,\frac{1}{6} \left(3^{-p}+5\right)\right)\right)\right)\\
=-3^n e^{\frac{1}{2} i \pi  3^n} \Phi'\left(-1,-1,\frac{1}{2}
   \left(3^{-n}+1\right)\right)+i \log \left(\frac{A^3}{\sqrt[3]{2} \sqrt[4]{e}}\right)\\
+\frac{\pi  \cos \left(\pi  3^n\right)+4\log \left(i 3^n\right) \left(3^n \cos \left(\frac{\pi  3^n}{2}\right)-i \sin \left(\frac{\pi  3^n}{2}\right)\right)-\pi }{8\left(\cos \left(\pi  3^n\right)-1\right)}
\end{multline}
\end{example}
\begin{proof}
Use equation (\ref{eq:theorem_cc}) and take the first derivative with respect to $k$ and set $k=1,m=\pi/2,a=e^{i}$ and simplify using equation (18) in \cite{guillera}.
\end{proof}
\begin{example}
Finite series involving Ap\'{e}ry's constant $\zeta(3)$.
\begin{multline}
\sum_{p=0}^{n-1}\csc ^3\left(\frac{1}{2} \pi  3^{p+1}\right) \left(4\times 9^{p+1} \Phi'\left(-1,-2,\frac{1}{6}
   \left(3^{-p}+1\right)\right)\right. \\ \left.
+4\times 9^{p+1} \Phi'\left(-1,-2,\frac{1}{6}
   \left(3^{-p}+5\right)\right)+\left(5\times 9^p-1\right) \log \left(i 3^{p+1}\right)\right)\\
=4\times 9^n \csc ^3\left(\frac{\pi 3^n}{2}\right) \Phi'\left(-1,-2,\frac{1}{2} \left(3^{-n}+1\right)\right)+\frac{1}{2}\left(9^n-1\right) \log \left(i 3^n\right) \csc ^3\left(\frac{\pi  3^n}{2}\right)\\
-\frac{7 \zeta (3)}{\pi ^2}
\end{multline}
\end{example}
\begin{proof}
Use equation (\ref{eq:theorem_cc}) and take the first derivative with respect to $k$ and set $k=2,m=\pi/2,a=0$ and simplify using equation (19) in \cite{guillera}.
\end{proof}
\begin{example}
Finite series involving Catalan's constant $C$.
\begin{multline}
\sum_{p=0}^{n-1}3^p e^{-\frac{5}{2} i \pi  3^p} \csc ^2\left(\frac{1}{2} \pi  3^{p+1}\right) \left(-12
   \left(\Phi'\left(-1,-1,\frac{1}{6}\right)+\Phi'\left(-1,-1,\frac{5}{6}\right)\right)\right. \\ \left.
+e^{\frac{5}{2} i \pi  3^p} \log \left(i 3^{p+1}\right) \left(5 \cos \left(\frac{\pi  3^p}{2}\right)+\cos
   \left(\frac{5 \pi  3^p}{2}\right)\right)\right)\\
=\frac{2 \left(4 C \left(3^n e^{\frac{1}{2} i \pi  3^n} \sin ^2\left(\frac{\pi 3^n}{2}\right)-i\right)+\pi  3^n \log \left(i 3^n\right) \cos \left(\frac{\pi  3^n}{2}\right)\right)}{\pi  \left(\cos \left(\pi3^n\right)-1\right)}
\end{multline}
\end{example}
\begin{proof}
Use equation (\ref{eq:theorem_cc}) and take the first derivative with respect to $k$ and set $k=1,m=\pi/2,a=0$ and simplify using equation (20) in \cite{guillera}.
\end{proof}
%
%
\section{Table of results: Part III}
In this section we evaluate Theorem (\ref{eq:theorem_ss1}) for various values of the parameters involved to derive related closed form formulae in terms of special functions and fundamental constants.
\begin{example}
Finite product of quotient gamma functions.
\begin{multline}
\prod_{p=0}^{n-1}\left(1-\frac{2\times 3^{p+1}}{a}\right)^{2\times 3^{-p-1}} \left(\frac{\Gamma \left(\frac{1}{12} \left(3^{-p}
   a-6\right)\right)}{\Gamma \left(\frac{1}{4} 3^{-p-1} a\right)}\right)^{2\times 3^{-p-1}}\\
 \left(\frac{3^{-\frac{1}{4} a
   3^{-p}-2 p+\frac{3}{2}} \left(a^2-4 a 3^{p+1}+32\times 9^p\right) \Gamma \left(\frac{3^{-p} a}{4}-3\right)}{\Gamma
   \left(\frac{1}{4} 3^{-p-1} a-1\right) \Gamma \left(\frac{1}{12} \left(3^{-p} a+2\right)\right) \Gamma
   \left(\frac{1}{12} \left(3^{-p} a+10\right)\right)}\right)^{1-3^{-p}}\\
=\frac{2^{\frac{3}{2} \left(2
   n+3^{1-n}-3\right)} 3^{\frac{1}{4} \left(-2 n-3^{1-n}+3\right)} \pi ^{\frac{3}{2} \left(1-3^{-n}\right)-n}
   \left(1-\frac{2\times 3^n}{a}\right)^{-3^{-n}} }{a}\\
   \left(a-2\times 3^n\right) \left(\frac{\Gamma \left(\frac{1}{4} \left(3^{-n}a-2\right)\right)}{\Gamma \left(\frac{3^{-n} a}{4}\right)}\right)^{1-3^{-n}}
\end{multline}
\end{example}
\begin{proof}
Use equation (\ref{eq:theorem_ss1}) and set $m=0$ and simplify in terms of the Hurwitz zeta function using entry (4) in the Table below equation (64:12:7) on page 692 in \cite{atlas}. Next take the first partial derivative with respect to $k$ and set $k=0$ and simplify in terms of the log-gamma function using equation (64:10:2) in \cite{atlas}. Finally take the exponential function of both sides and simplify both sides to yield the stated result.
\end{proof}
\begin{example}
Finite product of quotient gamma functions.
\begin{multline}
\prod_{p=0}^{n-1}\frac{\left(\left(a 3^{-p}-4\right) \left(a 3^{-p}-2\right) \left(a 3^{-p-1} \Gamma \left(\frac{1}{2} 3^{-p-1}a\right)\right)^{2/3} \Gamma \left(\frac{1}{6} \left(3^{-p} a-4\right)\right) \Gamma \left(\frac{1}{6}
   \left(3^{-p} a-2\right)\right)\right)^{1-3^{-p}}}{\Gamma \left(\frac{1}{2} 3^{-p-1} a\right)^{2/3}}\\
=2^{\frac{3}{2} \left(2 n+3^{1-n}-3\right)} \pi ^{\frac{3}{2} \left(3^{-n}-1\right)+n} 3^{\frac{1}{16} 9^{-n} \left(4\times 3^n \left(2\left(5\times 3^n-2\right) n-9 \left(3^n-1\right)\right)-3 a \left(-4\times 3^n+9^n+3\right)\right)-\frac{1}{3} n (n+1)}
   a^{2 n/3} \\
\left(a \Gamma \left(\frac{3^{-n} a}{2}\right)\right)^{3^{-n}-1}
\end{multline}
\end{example}
\begin{proof}
Use equation (\ref{eq:theorem_ss}) and set $m=\pi/2$ and simplify in terms of the Hurwitz zeta function using entry (4) in the Table below equation (64:12:7) on page 692 in \cite{atlas}. Next take the first partial derivative with respect to $k$ and set $k=0$ and simplify in terms of the log-gamma function using equation (64:10:2) in \cite{atlas}. Finally take the exponential function of both sides and simplify both sides to yield the stated result.
\end{proof}
%
%
%
%
%
\begin{example}
Functional equation for the Hurwtiz-Lerch zeta function.
\begin{multline}
\Phi (z,s,a)
=3^{-2 s-1} \left(3^s \left(3 \Phi \left(z^3,s,\frac{a}{3}\right)+z \left(3 \Phi
   \left(z^3,s,\frac{a+1}{3}\right)-z \Phi \left(z^3,s,\frac{a+2}{3}\right)\right)\right)\right. \\ \left.
+4 z^2 \left(\Phi
   \left(z^9,s,\frac{a+2}{9}\right)+z^6 \Phi \left(z^9,s,\frac{a+8}{9}\right)+z^3 \Phi
   \left(z^9,s,\frac{a+5}{9}\right)\right)\right)
\end{multline}
\end{example}
\begin{proof}
Use equation (\ref{eq:theorem_ss1}) and set $n=3,m=\frac{\log(-z)}{6i},a=6 (a i - 1),k=-s$ and simplify.
\end{proof}
\begin{example}
Finite product of quotient cosine functions. 
\begin{multline}
\prod_{p=0}^{n-1}\left(\frac{\left(-1+2 \cosh \left(2\times 3^p m\right)\right) \left(\frac{\cosh \left(3^p r\right)}{\cosh \left(3^p
   m\right)}\right)^2}{-1+2 \cosh \left(2\times 3^p r\right)}\right)^{3^{1-2 p} \left(-1+3^p\right)} \left(\frac{\cosh
   \left(3^{1+p} m\right)}{\cosh \left(3^{1+p} r\right)}\right)^{2\times 3^{-2 p}}\\
=\left(\frac{\cosh \left(3^n
   m\right)}{\cosh \left(3^n r\right)}\right)^{9^{1-n} \left(-1+3^n\right)}
\end{multline}
\end{example}
\begin{proof}
Use equation (\ref{eq:theorem_ss1}) and form a second equation by replacing $m\to r$. Next take the differnce of the two equations and simplify. Then set $k=-1,a=1$ and simplify in terms of the logarithm function using entry (5) in the Table below equation (64:12:7) on page 692 in \cite{atlas}. Next take the exponential function of both sides and simplify to yield the stated result.
\end{proof}
\begin{example}
Finite product of quotient cosine functions.
\begin{multline}
\prod_{p=0}^{n-1}\left(\frac{1-2 \cosh \left(2\times 3^p x\right)}{1-2 \cosh \left(2\times 3^{-1+p} x\right)}\right)^{\frac{1}{2} 9^{-p}
   \left(-1+3^{1+p}\right)} \left(\frac{\cosh \left(3^{-1+p} x\right)}{\cosh \left(3^p x\right)}\right)^{9^{-p}
   \left(-4+3^{1+p}\right)}\\
=\left(\frac{\cosh \left(3^n x\right)}{\cosh \left(3^{-1+n} x\right)}\right)^{\frac{1}{2}
   9^{1-n} \left(-1+3^n\right)}
\end{multline}
\end{example}
\begin{proof}
Use equation (\ref{eq:theorem_ss1}) and set $k=1,a=0,m=x$ and apply the method in section (8) in \cite{reyn_ejpam}.
\end{proof}
\subsection{Finite product involving polynomials}
Finite products involving polynomials are used in the study of the Thue-Morse constant \cite{zorin,allouche}, where this constant is not a badly approximable number. In his renowned 1936 publication, Alan Turing explored the concept of the computability of real numbers. In simpler terms, a real number, denoted as $\alpha$, is deemed computable when there exists a Turing machine capable of generating a rational approximation to $\alpha$, which is accurate within a margin of $2^{-i}$, where $i$ represents the input value. The finite product involving polynomials is used to study the transcendence of the Thue–Morse Number on page 388 in \cite{allouche}, where a proof using a theorem on analytic functions was employed.
\begin{example}
Finite product involving polynomial functions.
\begin{multline}\label{eq:finite_poly}
\prod_{p=0}^{n-1}\left(z^{3^p}+1\right)^{3^{-2 p-1}} \left(\frac{\sqrt{\left(z^{3^{p-1}}-1\right)
   z^{3^{p-1}}+1}}{z^{3^{p-1}}+1}\right)^{3^{-2 p} \left(3^p-1\right)}\\
   =\left(z^{3^{n-1}}+1\right)^{\frac{1}{2} 3^{1-2n} \left(3^n-1\right)}\\
   =\sum_{p=0}^{\infty}\binom{\frac{1}{2} 3^{1-2 n} \left(-1+3^n\right)}{p}
   \left(z^{3^{n-1}}\right)^{\frac{1}{2} 3^{1-2 n} \left(3^n-1\right)-p}
\end{multline}
\end{example}
\begin{proof}
Use equation (\ref{eq:theorem_ss1}) and set $a=0,m=\log(z)/(6i)$ and simplify in terms of the Polylogarithm function using equation (64:12:2) in \cite{atlas}. Next simplify the Polylogarithm function in terms of the Hurwitz zeta function using equation (6) in \cite{jonq}. Next take the limit of both sides as $k\to -1$ and simplify in terms of the log-gamma function using equation (64:10:2) in \cite{atlas}. Next take the exponential function of both sides and simplify.
\end{proof}
%
%
%
%
\begin{figure}[H]
\includegraphics[scale=0.8]{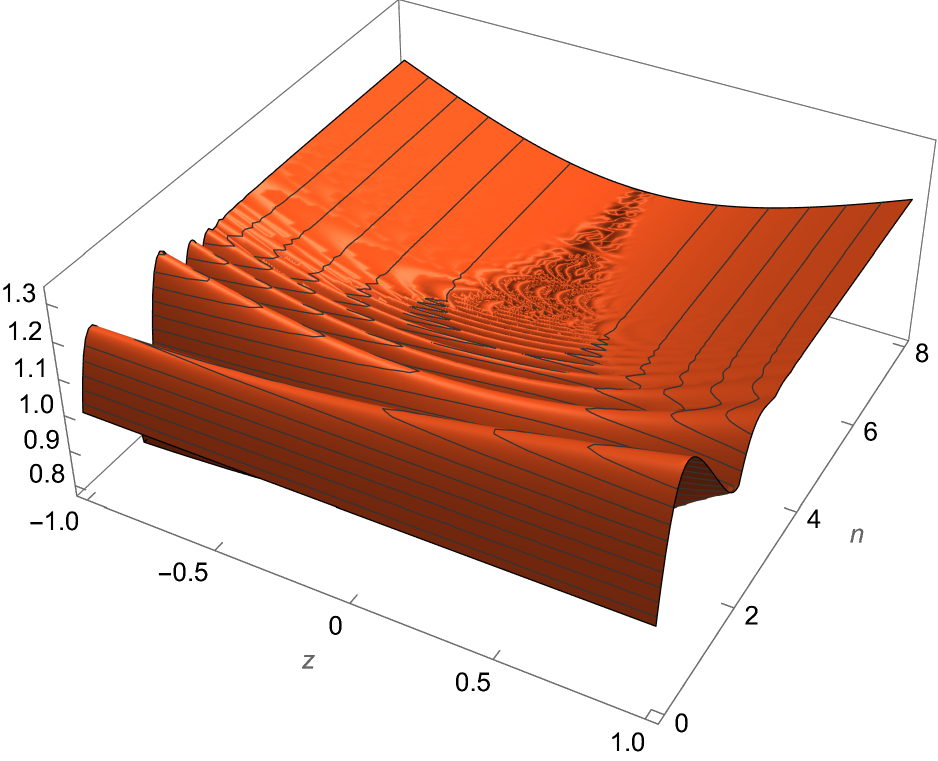}
\caption{Plot of  $f(z)=Re\left(\left(z^{3^{n-1}}+1\right)^{\frac{1}{2} 3^{1-2 n} \left(3^n-1\right)}\right)$, $z\in\mathbb{C}$.}
   \label{fig:fig2}
\end{figure}
\begin{figure}[H]
\includegraphics[scale=0.8]{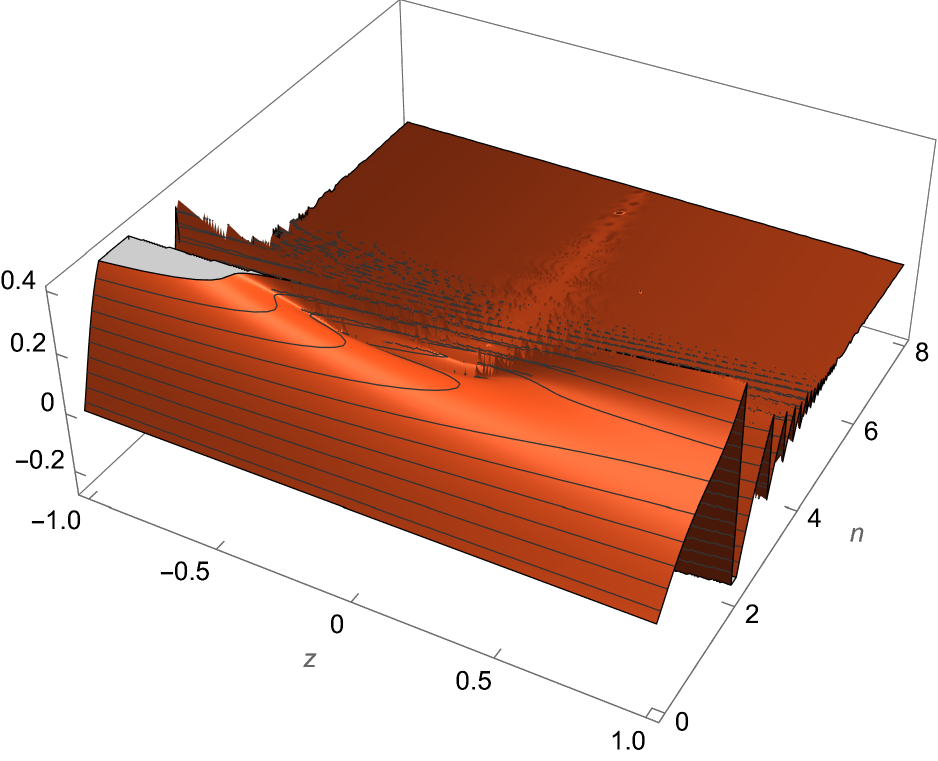}
\caption{Plot of  $f(z)=Im\left(\left(z^{3^{n-1}}+1\right)^{\frac{1}{2} 3^{1-2 n} \left(3^n-1\right)}\right)$, $z\in\mathbb{C}$.}
   \label{fig:fig2}
\end{figure}
%
\begin{figure}[H]
\includegraphics[scale=0.8]{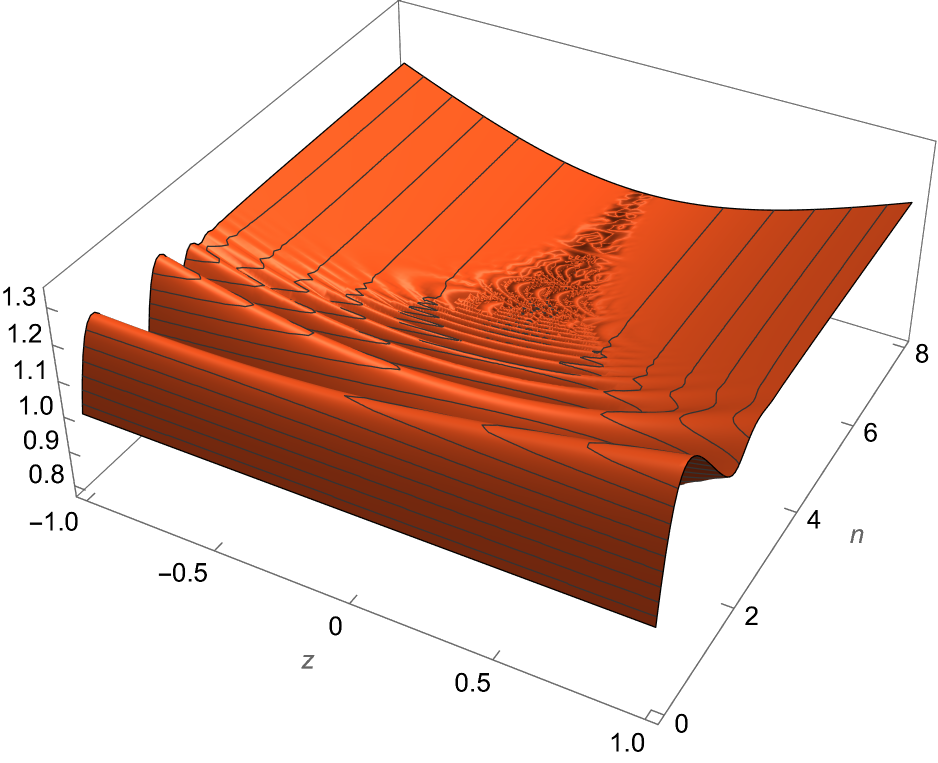}
\caption{Plot of  $f(z)=Abs\left(\left(z^{3^{n-1}}+1\right)^{\frac{1}{2} 3^{1-2 n} \left(3^n-1\right)}\right)$, $z\in\mathbb{C}$.}
   \label{fig:fig2}
\end{figure}
%
\begin{example}
Infinite product involving polynomial functions. 
\begin{equation}
\prod_{p=0}^{\infty}\left(z^{3^p}+1\right)^{3^{-2 p-1}} \left(\frac{\sqrt{\left(z^{3^{p-1}}-1\right)
   z^{3^{p-1}}+1}}{z^{3^{p-1}}+1}\right)^{3^{-2 p} \left(3^p-1\right)}=1
\end{equation}
\end{example}
\begin{proof}
Use equation (\ref{eq:finite_poly}) and take the limit of the right-hand side as $n\to \infty$ and simplify.
\end{proof}
\begin{example}
Finite sum of the digamma function in terms of Catalan's constant $C$. 
\begin{multline}
\sum_{p=0}^{n-1}27^{-p} \left(\psi ^{(1)}\left(1-\frac{1}{4} 3^{n-p-1}\right)+\frac{3}{2}
   \left(3^p-1\right) \left(\psi ^{(1)}\left(\frac{5}{6}-\frac{1}{4}
   3^{n-p-1}\right)-\psi ^{(1)}\left(\frac{1}{12}
   \left(4-3^{n-p}\right)\right)\right)\right. \\ \left.
-\psi ^{(1)}\left(\frac{1}{12}
   \left(6-3^{n-p}\right)\right)+\frac{3}{2} \left(3^p-1\right) \left(\psi
   ^{(1)}\left(\frac{1}{12} \left(2-3^{n-p}\right)\right)-\psi
   ^{(1)}\left(\frac{1}{12} \left(8-3^{n-p}\right)\right)\right)\right)\\
=-8 C
   3^{3-3 n} \left(3^n-1\right)
\end{multline}
\end{example}
\begin{proof}
Use equation (\ref{eq:theorem_ss1}) and set $k=-2,m=0,a=-3^n$ and simply using equation (21) in \cite{guillera}.
\end{proof}

\begin{example}
Finite series in terms of the  Glaisher-Kinkelin constant $A$. 
\begin{multline}
\sum_{p=0}^{n-1}\left(6 \left(3^p-1\right)
   \left(\Phi'\left(-1,-1,\frac{1}{3}\right)-\Phi'\left(-1,-1,\frac{2}{3}\right)\right)+3^p \log \left(i 3^{p+1}\right)\right)\\
=\frac{1}{6} \left(\left(-2 n+3^n-1\right) \log \left(\frac{16e^3}{A^{36}}\right)+3 \left(3^n-1\right) \log \left(i 3^n\right)\right)
\end{multline}
\end{example}
\begin{proof}
Use equation (\ref{eq:theorem_ss1}) and take the first derivative with respect to $k$ and set $k=1,m=0,a=0$ and simplify using equation (18) in \cite{guillera}.
\end{proof}
\begin{example}
Finite series in terms of the  Apery's constant $\zeta(3)$.
\begin{multline}
\sum_{p=0}^{n-1}3^p \left(3^p-1\right)
   \left(\Phi'\left(-1,-2,\frac{1}{3}\right)-\Phi'\left(-1,-2,\frac{2}{3}\right)\right)=-\frac{7 \left(-4\times 3^n+9^n+3\right) \zeta (3)}{36 \pi ^2}
\end{multline}
\end{example}
\begin{proof}
Use equation (\ref{eq:theorem_ss1}) and take the first derivative with respect to $k$ and set $k=2,m=0,a=0$ and simplify using equation (19) in \cite{guillera}.
\end{proof}
\begin{example}
Finite sum of the digamma function in terms of Catalan's constant $C$.
\begin{multline}
\sum_{p=0}^{n-1}\left(12 \Phi'\left(-1,-1,1-\frac{1}{2} 3^{n-p-1}\right)-18 \left(3^p-1\right)
   \Phi'\left(-1,-1,\frac{1}{6} \left(2-3^{n-p}\right)\right)\right. \\ \left.
+18 \left(3^p-1\right)
   \Phi'\left(-1,-1,\frac{1}{6} \left(4-3^{n-p}\right)\right)+3^{-p} \left(3^n-3^{2
   p+1}\right) \log \left(i 3^{p+1}\right)\right)\\
=\frac{6 C \left(3^n-1\right)}{\pi }
\end{multline}
\end{example}
\begin{proof}
Use equation (\ref{eq:theorem_ss1}) and take the first derivative with respect to $k$ and set $k=1,m=0,a=-3^n$ and simplify using equation (20) in \cite{guillera}.
\end{proof}
\newpage
\begin{center}
\begin{table}[h]
\setlength{\tabcolsep}{0.03pt} 
\renewcommand{\arraystretch}{2} 
\begin{tabular}{ c | c }
\hline
 $\prod\limits_{p=0}^{n-1}\left(\frac{9^{p+1} \Gamma \left(3^{-p-1} a+\frac{1}{6}\right) \Gamma
   \left(3^{-p-1} a+\frac{5}{6}\right) \left(\frac{\Gamma \left(3^{-p-1}
   a+\frac{1}{2}\right)}{\Gamma \left(3^{-p-1}
   a+1\right)}\right)^{2/3}}{\left(3^p-a\right) \left(2\times 3^p-a\right) \Gamma
   \left(\frac{1}{3} \left(3^{-p} a-2\right)\right) \Gamma \left(\frac{1}{3}
   \left(3^{-p} a-1\right)\right)}\right)^{3^{-p}}$ &  $\frac{3^{\frac{3}{4}
   \left(1-3^{-n}\right)} \Gamma \left(a+\frac{1}{2}\right) \left(\frac{\Gamma
   \left(3^{-n} a+1\right)}{\Gamma \left(3^{-n}
   a+\frac{1}{2}\right)}\right)^{3^{-n}}}{\Gamma (a+1)}$ \\ 
 $\prod\limits_{p=0}^{\infty}\left(\frac{9^{p+1} \Gamma \left(3^{-p-1} a+\frac{1}{6}\right) \Gamma
   \left(3^{-p-1} a+\frac{5}{6}\right) \left(\frac{\Gamma \left(3^{-p-1}
   a+\frac{1}{2}\right)}{\Gamma \left(3^{-p-1}
   a+1\right)}\right)^{2/3}}{\left(3^p-a\right) \left(2\times 3^p-a\right) \Gamma
   \left(\frac{1}{3} \left(3^{-p} a-2\right)\right) \Gamma \left(\frac{1}{3}
   \left(3^{-p} a-1\right)\right)}\right)^{3^{-p}}$ & $\frac{3^{3/4} \Gamma
   \left(a+\frac{1}{2}\right)}{\Gamma (a+1)}$  \\  
 $\prod\limits_{p=0}^{n-1} \left(\frac{\left(\frac{\cos \left(3^p m\right)}{\cos \left(3^p r\right)}\right)^{16} \left(1-2 \cos \left(2\times 3^p r\right)\right)^2}{\left(1-2 \cos \left(2\times 3^p m\right)\right)^2}\right)^{3^{-2
   p}}$ & $\left(\frac{\cos (m)}{\cos (r)}\right)^{18} \left(\frac{\cos \left(3^n r\right)}{\cos \left(3^n m\right)}\right)^{2\times 3^{2-2 n}}$\\ 
 $\prod\limits_{p=0}^{\infty}\left(\frac{\left(\frac{\cos \left(3^p m\right)}{\cos \left(3^p
   r\right)}\right)^{16} \left(1-2 \cos \left(2\times 3^p
   r\right)\right)^2}{\left(1-2 \cos \left(2\times 3^p
   m\right)\right)^2}\right)^{\frac{3^{-2 p}}{18}}$ & $\frac{\cos (m)}{\cos
   (r)}$\\ 
 $\prod\limits_{p=0}^{\infty}\left(\frac{3^{a 3^{-p}+2 p} \Gamma \left(3^{-p-1} a-1\right)}{\left(a^2-a
   3^{p+1}+2\times 9^p\right) \Gamma \left(3^{-p-1} a+1\right)^{2/3} \Gamma
   \left(3^{-p} a-3\right)}\right)^{3^{-p}}$ & $\frac{3^{\frac{9 a}{8}+3}}{\Gamma
   (a+1)}$\\ 
 $\prod\limits_{p=0}^{n-1}\frac{\Gamma \left(\frac{1}{6} \left(3^{-p} a+1\right)\right) \Gamma \left(\frac{1}{6} \left(3^{-p}
   a+5\right)\right)}{2 \pi }$ & $\frac{3^{-\frac{1}{4} a 3^{1-n} \left(3^n-1\right)} \Gamma
   \left(\frac{a+1}{2}\right)}{\Gamma \left(\frac{1}{2} \left(3^{-n} a+1\right)\right)}$\\ 
 $\prod\limits_{p=0}^{\infty}\frac{\Gamma \left(\frac{1}{6} \left(3^{-p} a+1\right)\right) \Gamma \left(\frac{1}{6} \left(3^{-p}
   a+5\right)\right)}{2 \pi }$ & $\frac{3^{-3 a/4} \Gamma \left(\frac{a+1}{2}\right)}{\sqrt{\pi }}$\\  
 $\prod\limits_{p=0}^{\infty}\left(\frac{\left(1+2 \cosh \left(2\times 3^p m\right)\right) \left(-1+2 \cosh \left(2\times 3^p r\right)\right)}{\left(-1+2 \cosh \left(2\times 3^p m\right)\right) \left(1+2 \cosh \left(2\times 3^p
   r\right)\right)}\right)^{3^{-p}} \left(\frac{\tanh \left(3^p r\right)}{\tanh \left(3^p m\right)}\right)^{2\times 3^{-p}}$ & $\left(\frac{\tanh (r)}{\tanh (m)}\right)^3$\\  
 $\prod\limits_{p=0}^{n-1}\left(\frac{\Gamma \left(\frac{1}{12} \left(3^{-p} a+7\right)\right) \Gamma \left(\frac{1}{12} \left(3^{-p}
   a+11\right)\right)}{\Gamma \left(\frac{1}{4} 3^{-p-1} \left(a+3^p\right)\right) \Gamma \left(\frac{1}{12}
   \left(3^{-p} a+5\right)\right)}\right)^{(-1)^p}$ & $\frac{3^{\frac{1}{4} \left((-1)^n-1\right)} \Gamma
   \left(\frac{a+3}{4}\right) \left(\frac{\Gamma \left(\frac{1}{4} \left(3^{-n} a+1\right)\right)}{\Gamma
   \left(\frac{1}{4} \left(3^{-n} a+3\right)\right)}\right)^{(-1)^n}}{\Gamma \left(\frac{a+1}{4}\right)}$\\  
 $\sum\limits_{p=0}^{n-1}9^{-p} \left(\psi ^{(1)}\left(\frac{1}{12} \left(2+3^{-p}\right)\right)+\psi ^{(1)}\left(\frac{1}{12} \left(10+3^{-p}\right)\right)\right)$ & $9 \left(-8 C-9^{-n} \psi
   ^{(1)}\left(\frac{1}{4} \left(2+3^{-n}\right)\right)+\pi ^2\right)$\\     
 $\prod\limits_{p=0}^{n-1}\left(\frac{\left(e^{-2 m 3^{p+1}}+1\right)^{2/3} \cosh ^2\left(m 3^p\right)}{2 \cosh \left(2 m 3^p\right)-1}\right)^{3^{-2
   p}}$ & $\frac{2^{\frac{3}{4} \left(1+3^{1-2 n}\right)} \cosh ^3(m)}{e^{3 m} \left(1+e^{-2 3^n
   m}\right)^{3^{1-2 n}}}$\\     
 $\prod\limits_{p=0}^{\infty}\left(\frac{\left(e^{-2 m 3^{p+1}}+1\right)^{2/3} \cosh ^2\left(m 3^p\right)}{2 \cosh \left(2 m 3^p\right)-1}\right)^{3^{-2
   p-1}}$ & $\frac{e^{-2 m}+1}{2^{3/4}}$\\   
 $\prod\limits_{p=0}^{n-1}\left(\frac{\left(-1+2 \cosh \left(2\times 3^p m\right)\right) \left(\frac{\cosh \left(3^p r\right)}{\cosh \left(3^p
   m\right)}\right)^2}{-1+2 \cosh \left(2\times 3^p r\right)}\right)^{3^{1-2 p} \left(-1+3^p\right)} \left(\frac{\cosh
   \left(3^{1+p} m\right)}{\cosh \left(3^{1+p} r\right)}\right)^{2\times 3^{-2 p}}$ & $\left(\frac{\cosh \left(3^n
   m\right)}{\cosh \left(3^n r\right)}\right)^{9^{1-n} \left(-1+3^n\right)}$\\ 
    $\prod\limits_{p=0}^{n-1}\left(\frac{1-2 \cosh \left(2\times 3^p x\right)}{1-2 \cosh \left(2\times 3^{-1+p} x\right)}\right)^{\frac{1}{2} 9^{-p}
   \left(-1+3^{1+p}\right)} \left(\frac{\cosh \left(3^{-1+p} x\right)}{\cosh \left(3^p x\right)}\right)^{9^{-p}
   \left(-4+3^{1+p}\right)}$ & $\left(\frac{\cosh \left(3^n x\right)}{\cosh \left(3^{-1+n} x\right)}\right)^{\frac{1}{2}
   9^{1-n} \left(-1+3^n\right)}$\\ 
 $\prod_{p=0}^{n-1}\left(z^{3^p}+1\right)^{3^{-2 p-1}} \left(\frac{\sqrt{\left(z^{3^{p-1}}-1\right)
   z^{3^{p-1}}+1}}{z^{3^{p-1}}+1}\right)^{3^{-2 p} \left(3^p-1\right)}$ & $\left(z^{3^{n-1}}+1\right)^{\frac{1}{2} 3^{1-2n} \left(3^n-1\right)}$\\ 
     \hline 
\end{tabular}
\caption{Table of sums and products}
\label{table:kysymys}
\end{table}
\end{center}
\newpage
%
%
%
\section{Conclusion}
In this article, we've introduced an approach for obtaining formulae for finite and infinite sums and products that incorporate the Hurwitz-Lerch zeta function. We've also explored intriguing specific instances through a combination of contour integration and established algebraic methods. Our intention is to employ these techniques in forthcoming research to derive additional formulae for other sums and products involving various special functions.
\end{document}